\documentclass{amsart}
\usepackage{amsthm,latexsym,amsfonts,amsmath,amssymb,mathrsfs,array,manfnt,stmaryrd}
\usepackage[pdftex]{graphicx}
\usepackage[all]{xy}
\usepackage{paralist}
\usepackage{framed}
\usepackage{cancel}
\usepackage{enumitem}
\usepackage{asymptote}
\usepackage{units}
\usepackage{gensymb}
\usepackage{setspace}
\usepackage{amscd}
\usepackage{microtype}
\usepackage{tikz}
\usetikzlibrary{matrix}
\usepackage{tikz-cd}

\newcommand{\into}{\hookrightarrow}

\newcommand{\abs}[1]{\left\lvert#1\right\rvert}

\newcommand{\N}{\ensuremath{\mathbb{N}}}
\newcommand{\Z}{\ensuremath{\mathbb{Z}}}

\newcommand{\C}{\ensuremath{\mathbb{C}}}

\newcommand{\A}{\ensuremath{\mathbb{A}}}

\newcommand{\M}{\mathcal{M}}
\newcommand{\Mbar}{\overline{\mathcal{M}}}

\renewcommand{\P}{\ensuremath{\mathbb{P}}}

\renewcommand{\O}{\ensuremath{\mathcal{O}}}

\DeclareMathOperator{\Hilb}{Hilb}

\DeclareMathOperator{\rk}{rk}

\DeclareMathOperator{\Span}{Span}

\DeclareMathOperator{\Gr}{Gr}

\DeclareMathOperator{\Hom}{Hom}

\DeclareMathOperator{\Sym}{Sym}

\DeclareMathOperator{\U}{U}

\usepackage{color}

\newcommand{\margincolor}{red}      
\definecolor{darkgreen}{rgb}{0,0.7,0}

\newcounter{margincounter}
\newcommand{\marginnum}{
\ifnum\value{margincounter}<10
\textcolor{\margincolor}{\begin{picture}(0,0)\put(2.2,2.4){\circle{9}}\end{picture}\footnotesize\arabic{margincounter}}
\else\ifnum\value{margincounter}<100
\textcolor{\margincolor}{\begin{picture}(0,0)\put(4.256,2.5){\circle{11}}\end{picture}\footnotesize\arabic{margincounter}}
\else
\textcolor{\margincolor}{\begin{picture}(0,0)\put(6.8,2.5){\circle{14}}\end{picture}\footnotesize\arabic{margincounter}}
\fi\fi
}

\theoremstyle{plain}
\newtheorem{theorem}{Theorem}
\numberwithin{theorem}{section}

\newtheorem{thm}[theorem]{Theorem}
\newtheorem*{thm*}{Theorem}

\newtheorem{prop}[theorem]{Proposition}
\newtheorem{obs}[theorem]{Observation}

\newtheorem{cor}[theorem]{Corollary}

\newtheorem{lem}[theorem]{Lemma}

\theoremstyle{definition}

\newtheorem{Def}[theorem]{Definition}

\theoremstyle{remark}

\newtheorem{rem}[theorem]{Remark}

\newtheorem{ex}[theorem]{Example}

\newtheorem{question}[theorem]{Question}
\newtheorem{notation}[theorem]{Notation}

\usepackage{tikz}
\usetikzlibrary{matrix}
\usepackage{fullpage}
\usepackage{bbm}

\DeclareMathOperator{\Mon}{Mon}

\DeclareMathOperator{\init}{in}

\DeclareMathOperator{\Trop}{Trop}
\usepackage{hyperref}
\renewcommand{\O}{\mathcal{O}}
\begin{document}
\title{The matroid stratification of the Hilbert scheme of points on
  $\P^1$} \date{\today} \author{Rob Silversmith}

\begin{abstract}
  Given a homogeneous ideal $I$ in a polynomial ring over a field, one
  may record, for each degree $d$ and for each polynomial $f\in I_d$,
  the set of monomials in $f$ with nonzero coefficients. These data
  collectively form the \textit{tropicalization} of $I$. Tropicalizing
  ideals induces a ``matroid stratification'' on any (multigraded)
  Hilbert scheme. Very little is known about the structure of these
  stratifications.


  In this paper, we explore many examples of matroid strata, including
  some with interesting combinatorial structure, and give a convenient
  way of visualizing them. We show that the matroid stratification in
  the Hilbert scheme of points $(\P^1)^{[k]}$ is generated by all
  Schur polynomials in $k$ variables. 
  We end with an application to the $T$-graph problem of
  $(\A^2)^{[n]};$ classifying this graph is a longstanding open
  problem, and we establish the existence of an infinite class of
  edges.
\end{abstract}
\maketitle

\section{Introduction}
Let $\mathbbm{k}$ be a field. The \textit{support} of a homogeneous
polynomial $f\in\mathbbm{k}[x_1,\ldots,x_r]$ is the set of monomials
with nonzero coefficient in $f$. Let
$I\subseteq R=\mathbbm{k}[x_1,\ldots,x_r]$ be a homogeneous ideal. For
each degree $d$, the data of all supports of polynomials in $I_d$ 
comprise a combinatorial portrait called the \textit{tropicalization}
of $I_d$, denoted $\Trop(I_d)$. A \textit{matroid} (see \cite{Oxley2006}
or Definition \ref{Def:Matroid}) is the data of a finite set $E$,
together with a subset of $M\subseteq2^E$ satisfying certain
combinatorial conditions. $\Trop(I_d)$ is an example of 
a matroid, where $E=\Mon_d$ is the set of degree $d$ monomials in
$x_1,\ldots,x_r$.

In this paper, we study $I$ via the infinite sequence of matroids
$\Trop(I)=(\Trop(I_d))_{d\ge0}$; this sequence is the
\textit{tropicalization} of $I$. The matroids satisfy a certain
combinatorial compatibility condition, namely the defining condition
of 
a \textit{tropical ideal} (Definition
\ref{Def:TropicalIdeal}). 

A (multigraded) Hilbert scheme is a moduli space parametrizing
homogeneous ideals. The fibers of the function $I\mapsto\Trop(I)$
define a ``matroid stratification'' on any Hilbert scheme, possibly
with countably many strata, analogous to, and generalizing, the more
well-known matroid stratification on
$\Gr(m,\mathbbm{k}^n)$. 

We identify the matroid stratification in the case of principal
homogeneous ideals in $\mathbbm{k}[x,y]$, i.e. in the Hilbert scheme
of points $(\P^1)^{[k]}$. Note that a symmetric polynomial in $k$
variables defines a divisor on $(\P^1)^{[k]}$ via the identification
$(\P^1)^{[k]}\cong\Sym^k(\P^1).$ Then: 

\medskip

\noindent\textbf{Theorem \ref{thm:SchurPolynomials}.} The matroid stratification on
$(\P^1)^{[k]}$ is the stratification generated by all Schur
polynomials $s_\lambda$ in $k$ variables.


 \medskip



 We end with an application to the \textit{$T$-graph problem} for
 $(\A^2)^{[n]}$, which was our original motivation for the
 project. Let $X$ be a variety with the action of an algebraic torus
 $T$ such that the fixed point set $X^T$ is finite. The $T$-graph of
 $X$ is a graph with vertex set $X^T$, and an edge between two fixed
 points if they are the two limit points of a 1-dimensional
 $T$-orbit. A Hilbert scheme has a $T=(\C^*)^r$-action by scaling the
 variables $x_1,\ldots,x_r.$ Determining the $T$-graphs of Hilbert
 schemes is a difficult problem that has been studied by Iarrobino,
 Evain, Altmann and Sturmfels, Hering and Maclagan, and others
 \cite{Iarrobino1972,Evain2004,AltmannSturmfels2005,HeringMaclagan2012}. We
 show:

\medskip

\noindent \textbf{Theorem \ref{thm:Necklace3}.} Let $\mathbbm{k}=\C.$
Let $k\ge1$ and $d>k.$ Let $S$ be the set of 1-dimensional
$(\C^*)^2$-orbits in $(\A^2)^{[dk]}$ whose limit points are the two
fixed points $(x^k,y^d)$ and $(x^d,y^k).$ Then $S$ is a finite set, in
natural bijection with the set of binary necklaces with $k$ black and
$d-k$ white beads. (In particular, $(x^k,y^d)$ and $(x^d,y^k)$ are
connected by an edge in the $T$-graph of $(\A^2)^{[dk]}$.)






\medskip

In Section \ref{sec:NecklaceToMatroid}, we pose some easily-stated
questions from combinatorial linear algebra that we cannot answer. The
answers would elucidate the relationship between Theorem 
\ref{thm:Necklace3} and Theorem \ref{thm:SchurPolynomials}.




\medskip


\noindent\textbf{Relation to other work.} The forthcoming paper \cite{FinkGiansiracusaGiansiracusa2019} of
Fink-Giansiracusa-Giansiracusa is closely related to this
one. Motivated by understanding ``tropical Hilbert schemes,'' which
are moduli spaces of tropical ideals over arbitrary valued fields,
they also investigate the tropicalizations of ideals of points in
$\P^1.$ Our results complement each other: this paper considers
trivially valued fields, and Hilbert schemes of arbitrarily many
points on $\P^1$, while they consider arbitrary valued fields, but
have results mainly for $\le2$ points in $\P^1$. We hope that these
perspectives can be merged to describe tropical Hilbert schemes of
arbitrarily many points in $\P^1$.

Zajaczkowska's Ph.D thesis \cite{Zajaczkowska2018} studied the
tropical Hilbert schemes of hypersurfaces of degrees 1 and 2 in $\P^1$
and $\P^2$. Among other things, the thesis contains the case $k=2$ of
Corollary \ref{thm:Necklace3}.

\subsection*{Acknowledgements} I am grateful to Diane Maclagan, for
guidance and information about the $T$-graph problem, for pointing out
that the sequences of matroids studied in this paper are a special
case of tropical ideals, for referring me to known results in the
field
\cite{MaclaganRincon2018,FinkGiansiracusaGiansiracusa2019,Zajaczkowska2018},
for giving feedback, and for writing the useful and convenient
\texttt{TEdges} package for Macaulay2. I am also grateful to Tim Ryan
--- we jointly generated data that led to Corollary
\ref{thm:Necklace3}. I would also like to thank Noah Giansiracusa for
introducing me to \cite{FinkGiansiracusaGiansiracusa2019}, Rohini
Ramadas for helpful conversations about tropical geometry, and the
anonymous referee for suggesting many improvements. This project
was supported by NSF DMS-1645877, and by postdoctoral positions at the
Simons Center for Geometry and Physics and at Northeastern University.

\section{Multigraded Hilbert schemes and their matroid
  stratifications}\label{sec:TCurves}
\subsection{Multigraded Hilbert schemes}\label{sec:Multigraded}
Multigraded Hilbert schemes are the natural moduli spaces of
homogeneous ideals in a polynomial ring. Let $\mathbbm{k}$ be a field,
and consider the polynomial ring $R=\mathbbm{k}[x_1,\ldots,x_r]$.
\begin{Def}
  For $b\in\Z_{>0},$ a (positive)
  $\Z^b$-\textit{multigrading}\footnote{There is a more general notion
    of multigrading that we will not need, see
    \cite{HaimanSturmfels2004}.}
  $\mathbf{a}=(\vec a_1,\ldots,\vec a_r)$ on $R$ is an assignment of a
  \textit{multidegree}
  $\vec a_i\in\Z_{\ge0}^b\setminus\{(0,\ldots,0)\}$ to each variable
  $x_i$. A multigrading is \textit{nondegenerate} if the rowspan of
  $\mathbf{a}$ is a rank-$b$ lattice in $\Z^r.$
\end{Def}
All multigradings from now on are assumed to be nondegenerate. A
$\Z^b$-multigrading defines a decomposition
$R=\bigoplus_{d\in\Z_{\ge0}^b}R_d$. Any $\mathbf{a}$-homogeneous ideal
$I\subseteq R$ has a \textit{multigraded Hilbert
  function}\footnote{Positivity of $\mathbf{a}$ is necessary here;
  otherwise $R_d/(I\cap R_d)$ need not be finite-dimensional.}
$h:\Z^b_{\ge0}\to\N$, defined by
$h(d)=\dim_{\mathbbm{k}}(R_d/(I\cap R_d))$.

Haiman and Sturmfels \cite{HaimanSturmfels2004} define a
\textit{multigraded Hilbert scheme} $\Hilb^h_{\mathbf{a}}(\A^r)$ that
is a projective fine moduli space for $\mathbf{a}$-homogeneous ideals
with multigraded Hilbert function $h$. For each $d\in\Z_{\ge0}^b,$
there is a short exact sequence of vector bundles on
$\Hilb^h_{\mathbf{a}}(\A^r)$:
\begin{align}\label{eq:TautSeqD}
  0\to\mathcal{I}_d\to\mathcal{R}_d\to\mathcal{Q}_d\to0,
\end{align}
where $\mathcal{I}_d$ is the universal ideal sheaf, $\mathcal{R}_d$
denotes the trivial sheaf with fiber $R_d$, and $\mathcal{Q}_d$ is the
rank-$h(d)$ universal quotient sheaf.
\begin{ex}\label{ex:FiniteEmbed}
  An important special case is when $I$ has finite colength,
  i.e. $\dim_{\mathbbm{k}}(R/I)=\sum_{d\in\Z_{\ge0}^b}h(d)<\infty$. In
  this case $V(I)$ has finite length, and there is a natural embedding
  $\Hilb^h_{\mathbf{a}}(\A^r)\into(\A^r)^{\left[\sum_dh(d)\right]}$
  into the Hilbert scheme of points in $\A^r$.
\end{ex}
\begin{ex}
  When $r=2$, $\Hilb^h_{\mathbf{a}}(\A^2)$ is smooth, irreducible, and
  rational \cite{MaclaganSmith2010}, see also
  \cite{Iarrobino1972,Evain2004}.
\end{ex}
\begin{ex}
  If $b=1$ and $\sum_dh(d)$ is \textit{not} finite, then
  $\Hilb^h_{(a_1,\ldots,a_r)}(\A^r)$ has a natural map to a Hilbert
  scheme of subschemes of the weighted projective space
  $\P(a_1,\ldots,a_r),$ cut out by the same ideal.
  This map need not be an embedding, essentially due to the fact that
  $I\in\Hilb^h_{\mathbf{a}}(\A^r)$ could have $(x_1,\ldots,x_r)$ as an
  embedded prime.
\end{ex}
\subsection{Tropicalizing ideals}\label{sec:Trop}
Tropical geometry usually takes place over a valued field, but in this
paper we will always assume $\mathbbm{k}$ is trivially valued. We
present the definitions we need only in this simpler context; see
\cite{MaclaganRincon2018} for the general definitions.

First we briefly recall the basics of matroid theory. See
\cite{Oxley2006} for details, including how to reconcile the following
definition with the allusion in the introduction.
\begin{Def}\label{Def:Matroid}
  A \textit{matroid} $M=(E,r)$ is the data of a finite set $E$, called
  the \textit{groundset}, together with a function $r:2^E\to\Z_{\ge0}$
  (where $2^E$ is the power set of $E$) called the \textit{rank
    function}, such that:
  \begin{enumerate}
  \item $r(\emptyset)=0,$
  \item For all subsets $S,S'\subseteq E,$
    $r(S\cup S')+r(S\cap S')\le r(S)+r(S')$, and
  \item For every subset $S\subseteq E$ and every element
    $x\in E\setminus S,$ $r(S)\le r(S\cup\{x\})\le r(S)+1.$
  \end{enumerate}
  The \textit{rank} of $M$ is $r(E).$ A subset $S\subseteq E$ is
  called \textit{dependent} if $r(S)<\abs{S}$, and
  \textit{independent} otherwise. A maximal independent subset is
  called a \textit{basis}, and all bases have cardinality $r(E)$. A
  minimal dependent subset is called a \textit{circuit}, and a union
  of circuits is called a \textit{cycle}. A 1-element circuit is
  called a \textit{loop}, and an element of $E$ not contained in any
  dependent sets is a \textit{coloop}. The \textit{corank function}
  is $r^*(S)=\abs{S}-r(S).$ A subspace $V\subseteq\mathbbm{k}^E$ gives
  rise to a matroid $\Trop(V)$ with groundset $E$ called its
  \textit{tropicalization}, with rank function
  $r(S)=\dim(\mathbbm{k}^S/V\cap\mathbbm{k}^S)$ for $S\subseteq
  E$. (Note that this is dual to some definitions in the literature.)
\end{Def}
\begin{ex}\label{ex:UniformMatroid}
  If $\mathbbm{k}$ is algebraically closed, the tropicalization of a
  generic dimension-$k$ subspace $V\in\Gr(k,\mathbbm{k}^E)$ is the
  \textit{uniform matroid} $U_{k,E}$, defined by the rank function
  \begin{align*}
    r(S)=
    \begin{cases}
      \abs{S}&\abs{S}\le k\\
      k&\abs{S}\ge k.
    \end{cases}
  \end{align*}
\end{ex}
We will use the following two standard facts.
\begin{lem}\label{lem:UnionOfCircuits}
  Let $V\subseteq\mathbbm{k}^E$ be a subspace.
  \begin{itemize}
  \item If $S\subseteq E$ is a circuit in $\Trop(V)$, then there
    exists $v=(v_e)_{e\in E}\in V$ such that $S=\{e\in E:v_e\ne0\}.$
  \item For any $v=(v_e)_{e\in E}\in V$, the set $S=\{e\in E:v_e\ne0\}$ is
    a cycle in $\Trop(V).$
  \end{itemize}
(Over an
  infinite field, the converse of the second statement holds.)
\end{lem}
Now we introduce our main objects of study.
\begin{Def}\label{Def:TropicalIdeal}
  Let $\mathbf{a}=(\vec a_1,\ldots,\vec a_r)$ be a positive
  multigrading on $k[x_1,\ldots,x_r].$ Let $\Mon_d(\mathbf{a})$ denote
  the set of monomials of degree $d$ with respect to the grading
  $\mathbf{a}$. A \textit{tropical (homogeneous) ideal
    $\mathscr{M}=(\mathscr{M}_d)_{d\in\Z_{\ge0}^b}$ with respect to
    the grading $\mathbf{a}$} (over the Boolean semifield) is the data
  of, for each $d\in\Z_{\ge0}^b,$ a matroid
  $\mathscr{M}_d=(\Mon_d(\mathbf{a}),r_d),$ such that
    for any circuit $S$ of $\mathscr{M}_d,$ and any monomial
    $m'\in\Mon_{d'}(\mathbf{a})$, $m'S$ is a cycle in
    $\mathscr{M}_{d+d'}(\mathbf{a})$.
    The \textit{multigraded Hilbert function} of a tropical
    homogeneous ideal $\mathscr{M}$ is the function
    $d\mapsto r_d(\Mon_d(\mathbf{a})).$
\end{Def}
Just as a subspace of $\mathbbm{k}^n$ gives rise to a matroid (a
``tropical linear space over the Boolean semifield''), a homogeneous
ideal with respect to the grading $\mathbf{a}$ gives rise to a
tropical homogeneous ideal with grading $\mathbf{a}$:
\begin{Def}\label{Def:Tropicalization}
  Let $I\subseteq\mathbbm{k}[x_1,\ldots,x_r]$ be
  $\mathbf{a}$-homogeneous. The \textit{tropicalization} of $I$ is
  $\Trop(I)=(\Trop(I)_d)_{d\in\Z_{\ge0}^b}$, where
  $\Trop(I)_d=\Trop(I_d)$.
\end{Def}
Observe that Lemma \ref{lem:UnionOfCircuits} implies that $\Trop(I)$ satisfies the
condition in Definition \ref{Def:TropicalIdeal}, and that the
multigraded Hilbert functions of $I$ and $\Trop(I)$ agree by
definition.

\subsection{Pictures of tropical ideals}\label{sec:Pictures}
When $r=2,$ we visualize a tropical ideal $\mathscr{M}$ as follows. We
draw a grid whose boxes representing monomials in two variables $x$
and $y$, where the bottom-leftmost square represents the monomial
1. We draw each circuit of $\mathscr{M}$ as a line segment connecting
a collection of dots in the grid; these dots correspond to monomials
in the circuit. We also label each degree $d$ by the Hilbert function
of $\mathscr{M}$, evaluated at $d$. (For simplicity, all examples
shown have the standard grading $\mathbf{a}=(1,1)$.)

To avoid clutter, we may omit a circuit $S$ of $\mathscr{M}_d$ if we
deem it ``uninformative,'' i.e. if $S$ is ``forced'' to be dependent
by the existence of a circuit in lower degree. Precisely, from now on
we omit a circuit $S$ in degree $d$ if there exists a circuit $S'$ in
degree $d'<d$ and a collection $T$ of degree-$(d-d')$ monomials such
that $S\subseteq\bigcup_{m\in T}mS'$ and
$\abs{S}>\abs{\bigcup_{m\in T}mS'}-\abs{T}.$ In this case, $S$ must be
dependent, as follows.

Consider the ordering $\preceq$ on $\Mon_d(\mathbf{a})$ by increasing
$y$-exponent. By Definition \ref{Def:TropicalIdeal}, each set $mS'$ is
a cycle. For each $m\in T$, select a circuit of $mS'$ that contains
the $\preceq$-minimal element of $mS'$. Then a circuit elimination
argument, exactly analogous to matrix row-reduction, shows that the
set $\bigcup_{m\in T}mS'$ has corank at least $\abs{T}$.
\begin{ex}\label{ex:OmitCircuits}
  The ideal $I=(x^3+x^2y+2xy^2+3y^3,x^5,xy^4)$ has tropicalization
  pictured in Figure \ref{fig:pictures}, where in the left image all
  circuits are drawn, and in the right image uninformative circuits
  are omitted.
  \begin{figure}[h]
    \centering
    \includegraphics[height=2in]{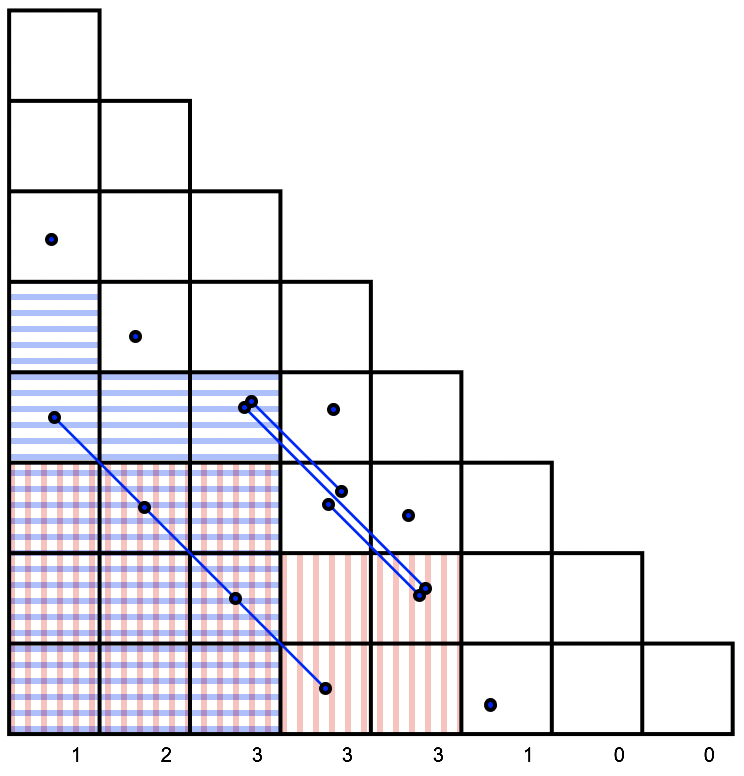}\quad\quad\quad\quad\includegraphics[height=2in]{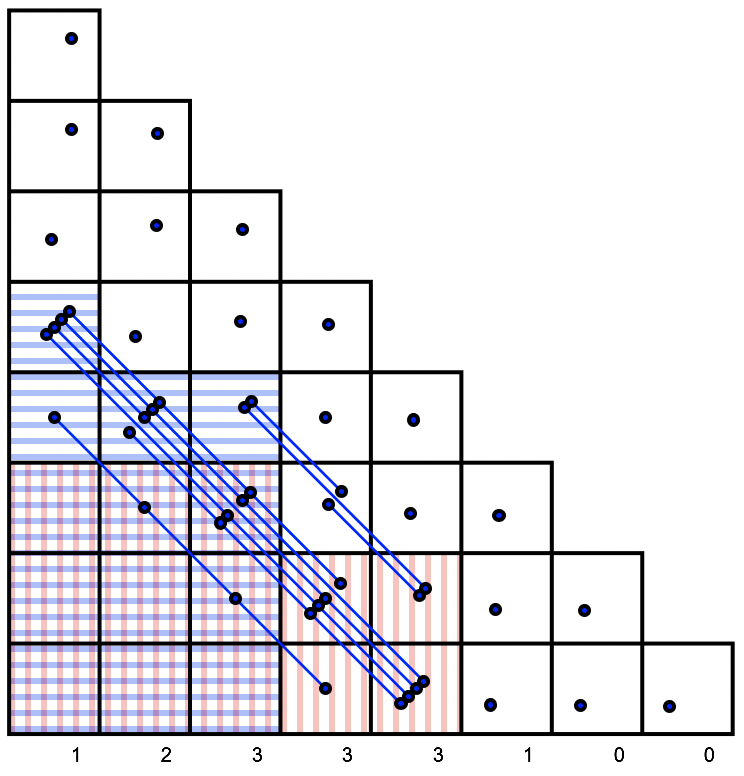}
    \caption{Two pictures of
      $\Trop(x^3+x^2y+2xy^2+3y^3,x^5,xy^4)$. See Notation
      \ref{not:ColorCode} for an explanation of the colors.}
    \label{fig:pictures}
  \end{figure}
\end{ex}

\subsection{Dependence loci and the matroid stratification} The
operation of tropicalization defines a stratification of any
multigraded Hilbert scheme $\Hilb_{\mathbf{a}}^h(\A^r)$, as
follows. Fix $d\in\Z_{\ge0}^b$ and $U\subseteq\Mon_d(\mathbf{a})$,
with $\ell:=\abs{U}$. We give a scheme-theoretic restatement of the
condition on $I\in\Hilb_{\mathbf{a}}^h(\A^r)$ that $U$ be dependent in
$\Trop(I)_d$. Consider the tautological sequence \eqref{eq:TautSeqD}
on $\Hilb_{\mathbf{a}}^h(\A^r).$ The collection $U$ defines, up to
sign, an element of $\bigwedge^\ell R_d.$ The wedge power of the map
$R_d\to\mathcal{Q}_d$ gives a global section $\sigma_U$ of
$\bigwedge^\ell\mathcal{Q}_d.$ This section vanishes at $I$ if and
only if the monomials in $U$ are linearly dependent modulo $I_d$,
i.e. if and only if $U$ is a dependent set in $\Trop(I_d).$ Thus we
define:
\begin{Def}
  The \textit{dependence scheme} of $U$ is 
  $$\mathcal{D}(U):=V(\sigma_U)\subseteq\Hilb_{\mathbf{a}}^h(\A^r).$$
\end{Def}
It is immediate that dependence schemes are closed subschemes. Since
matroids are uniquely defined by their dependent sets, we define:
\begin{Def}
  Let $\mathscr{M}$ be a tropical ideal. The \textit{matroid stratum}
  $\mathcal{S}(\mathscr{M})\subseteq\Hilb_{\mathbf{a}}^h(\A^r)$ of
  $\mathscr{M}$ is the locally closed subscheme
  \begin{align*}
    \bigcap_{d\in\Z_{\ge0}^b}\left(\bigcap_{\text{$U$ dependent in $\mathscr{M}_d$}}\mathcal{D}(U)\cap\bigcap_{\text{$U$ independent in $\mathscr{M}_d$}}\mathcal{D}(U)^C\right).
  \end{align*}
\end{Def}
Note that each stratum involves an infinite intersection of
Zariski-open sets, and therefore $\mathcal{S}(\mathscr{M})$ may not be
Zariski-locally closed. Indeed we will see in Section \ref{sec:P1}
that this does occur! However, if $\sum_{d\in\Z_{\ge0}^b}h(d)<\infty,$
then there \textit{are} no independent sets in sufficiently large
degree --- this implies there are finitely many strata in the matroid
stratification of $\Hilb^h_{\mathbf{a}}(\A^r)$, and they are
Zariski-locally closed.

\begin{rem}
  The number of strata in the matroid stratification of
  $\Hilb_{\mathbf{a}}^h(\A^r)$ is countable, as follows. A stratum
  $\mathcal{S}(\mathscr{M})$ is determined by the collection of sets
  $U$ such that $\mathcal{D}(U)\supseteq\mathcal{S}(\mathscr{M});$ in
  particular, $\mathcal{S}(\mathscr{M})$ is the unique stratum whose
  Zariski closure is
  $\bigcap_{d}\bigcap_{\substack{U\subseteq\Mon_d(\mathbf{a})\\\mathcal{D}(U)\supseteq\mathcal{S}(\mathscr{M})}}\mathcal{D}(U).$
  As $\Hilb_{\mathbf{a}}^h(\A^r)$ is Noetherian, any such intersection
  is actually finite; this defines an injective function from the set
  of matroid strata in $\Hilb_{\mathbf{a}}^h(\A^r)$ into the set of
  finite intersections of the countable collection
  $\{\mathcal{D}(U)\}_U$ of subsets.

  Note, however, that this argument does not imply that the number of
  tropical \textit{ideals} with fixed grading and Hilbert function is
  countable; indeed, we do not know whether this is the case.
\end{rem}

\begin{ex}\label{ex:Necklace1}
  We here introduce a simple, but surprising, example of a dependence
  locus, which we will return to repeatedly. Assume
  $\mathbbm{k}=\C$. Let $r=2$, $b=1,$ and $\mathbf{a}=(1,1)$, and
  suppose there exists $k\ge1$ such that
  \begin{align}\label{eq:HilbPrin}
    h(d)=
    \begin{cases}
      d+1&d<k\\
      k&d\ge k.
    \end{cases}
  \end{align}
  The corresponding Hilbert scheme is the moduli space of
  \textbf{principal} homogeneous ideals in $\mathbbm{k}[x,y]$
  generated in degree $k$, i.e. the Hilbert scheme $(\P^1)^{[k]}$ of
  length-$k$ subschemes of $\P^1.$ Fix $d\ge k,$ and let
  $U=\{x^d,y^d\}.$ We classify $\mathcal{D}(U)\subseteq(\P^1)^{[k]}.$

  Suppose $I=(f)\in\mathcal{D}(U)\subseteq(\P^1)^{[k]}$. Then $(f)$
  contains a polynomial of the form $c_1x^d+c_2y^d,$ i.e. there exists
  a degree-$(d-k)$ polynomial $p$ such that $pf=c_1x^d+c_2y^d.$ Note
  that $V(pf)$ consists of the $d$ points
  $\{[\zeta:1]:\zeta^d=c_2/c_1\}$ (as long as $c_1\ne0$). These are
  the vertices of a regular $d$-gon in $\C$ centered at the
  origin. Since $V(f)$ is a length-$k$ subscheme of $V(pf)$, $V(f)$
  consists of $k$ of the $d$ vertices.

  Conversely, given a collection of $k$ points $z_1,\ldots,z_k\in\C$
  that are distinct vertices of some regular $d$-gon centered at 0,
  the defining polynomial $f$ of $\{z_1,\ldots,z_k\}$ satisfies
  $(f)\in\mathcal{D}(U)$ (simply by letting $p$ be the defining
  polynomial of the other $d-k$ vertices).

  To visualize $\mathcal{D}(U)$ further, consider the $\C^*$-action on
  regular $d$-gons centered at 0. This defines an action on
  $\mathcal{D}(U),$ and the collection of ratios
  $z_2/z_1,\ldots,z_k/z_1$ defines an orbit; this collection is
  equivalent to the data of a binary necklace with $k$ black and $d-k$
  white beads. Let $N_{d,k}$ denote the set of such necklaces. Then
  $\mathcal{D}(U)$ is a union of rational curves indexed by $N_{d,k}$,
  all of which intersect at the two points $(x^k)$ (where the $d$-gon
  is scaled down to 0) and $(y^k)$ (where the $d$-gon is scaled out to
  $\infty$).
\end{ex}

In a rank-$k$ matroid $(E,r)$, a set $S\subseteq E$ with
$\abs{S}\le k$ is dependent if and only if $S'$ is dependent for every
$S'\supseteq S$ with $\abs{S'}=k$. (This follows from the fact that
all bases have cardinality $k$.) We have the following
scheme-theoretic version of this fact, which we will use in Section
\ref{sec:P1}:
\begin{prop}\label{prop:IntersectDependenceLoci}
  Fix a graded Hilbert function $h$. Let
  $U\subseteq\Mon_d(\mathbf{a})$ with $\abs{U}\le h(d)$.  The
  dependence scheme
  $\mathcal{D}(U)\subseteq\Hilb_{\mathbf{a}}^h(\A^r)$ satisfies
  $$\mathcal{D}(U)=\bigcap_{\substack{W\supseteq
      U\\\abs{W}=h(d)}}\mathcal{D}(W).$$ (Note: This also holds if
  $\abs{U}>h(d),$ in which case both sides are equal to
  $\Hilb_{\mathbf{a}}^h(\A^r)$.)
\end{prop}


\begin{proof}
  Consider the sequence of maps
  \begin{align*}
    \bigwedge^{\abs{U}}\Span(U)&\xrightarrow{\iota}{}\bigwedge^{\abs{U}}\mathcal{Q}_d\xrightarrow{w}{}\bigoplus_{U'\in\binom{\Mon_d}{h(d)-\abs{U}}}\bigwedge^{h(d)}\mathcal{Q}_d\xrightarrow{p}{}\bigoplus_{\substack{U'\in\binom{\Mon_d}{{h(d)}-\abs{U}}\\U'\cap
    U\ne\emptyset}}\bigwedge^{h(d)}\mathcal{Q}_d,
  \end{align*}
  where $\iota$ is the inclusion, $p$ is the projection, and
  $w(\alpha)=\alpha\wedge\bigwedge_{u\in U'}u$. Note that $w$ is
  injective, as it is induced by the nondegenerate pairing
  $\bigwedge^{\abs{U}}\mathcal{Q}_d\otimes\bigwedge^{{h(d)}-\abs{U}}\mathcal{Q}_d\to\mathbbm{k}$. Thus
  $V(w(\sigma_U))=V(\sigma_U).$

  Also, $p\circ w\circ\iota$ is zero, so $w\circ\iota$ factors through
  $\ker(p)=\bigoplus_{\substack{U'\in\binom{\Mon_d}{{h(d)}-\abs{U}}\\U'\cap
      U=\emptyset}}\bigwedge^{h(d)}\mathcal{Q}_d.$ Thus
  \begin{align*}
    \mathcal{D}(U)=V(\sigma_U)=V(w(\sigma_U))&=\bigcap_{\substack{U'\in\binom{\Mon_d}{{h(d)}-\abs{U}}\\U'\cap
    U=\emptyset}}V(\sigma_{U\cup
    U'})=\bigcap_{\substack{U'\in\binom{\Mon_d}{{h(d)}-\abs{U}}\\U'\cap
    U=\emptyset}}\mathcal{D}(U'\cup U).\qedhere
  \end{align*}
\end{proof}
The matroid stratification of $\Hilb_{\mathbf{a}}^h(\A^r)$ satisfies
the following straightforward recursivity relation, which implies that
when studying strata, we may ignore ideals of the form $I=mI'$, where
$m$ is a monomial. There are natural inclusions
$\iota_1,\ldots,\iota_r$ between $\mathbf{a}$-multigraded Hilbert
schemes, defined by $\iota_i(I)=x_iI.$
\begin{prop}\label{prop:StrataPullback}
  Let $U\subseteq\Mon_d(\mathbf{a})$. Then
  $$\iota_i^{-1}(\mathcal{D}(U))=\mathcal{D}\left(\frac{1}{x_i}(U\setminus\{m\in
    U:x_i\nmid m\})\right).$$
\end{prop}
We omit the proof, as it is straightforward, and we will use the
Proposition only to reduce the number of strata that are of interest.

\section{The matroid stratification of $(\P^1)^{[k]}$}\label{sec:P1}

Let $k>0.$ In this section, we will describe (Theorem
\ref{thm:SchurPolynomials}) the matroid stratification of
$\Hilb_{\mathbf{a}}^h(\A^r)$ in the case $r=2,$ $b=1$,
$\mathbf{a}=(1,1),$ and let $h$ be as in \eqref{eq:HilbPrin}. We write
$R=\mathbbm{k}[x,y]$. Note that $\Hilb_{\mathbf{a}}^h(\A^r)$ is simply
the familiar Hilbert scheme of points $(\P^1)^{[k]}$.  Recall that
\begin{align*}
  (\P^1)^{[k]}\cong\Sym^k(\P^1)\cong\P^k,
\end{align*}
where $[a_0:a_1:\cdots:a_k]\in\P^k$ corresponds the principal ideal
$I=(a_0x^k+a_1x^{k-1}y+\cdots+a_ky^k)\in(\P^1)^{[k]}$, and to the set
of roots (with multiplicity) $V(I)\in\Sym^k(\P^1).$ 

We will describe the matroid stratification on $(\P^1)^{[k]}$ via
vanishing loci of sections of line bundles. For convenience, we note
that the sheaf $\O(n)$ on $\P^k$ is identified with the sheaf
$\Sym^k(\P^1)$ of $S_k$-symmetric functions in $k$ pairs of variables
$x_1,y_1,\ldots,x_k,y_k$, bihomogeneous in each pair of variables of
degree $n$. Also, the tautological line bundle $\mathcal{I}_k$ on
$(\P^1)^{[k]}$ (see \eqref{eq:TautSeqD}) is identified with the bundle
$\O(-1)$ on $\P^k$.


\subsection{The correspondence between Schur polynomials and subsets
  of monomials}
We introduce the following version of Schur polynomials, bihomogenized
in each variable.
\begin{Def}\label{Def:SymmetricPolynomials}
  Let $k\ge1.$ Let $\lambda=(\lambda_1,\ldots,\lambda_m)$ be a
  partition, in nonincreasing order, with $c\le k$ parts. We write
  $\lambda_i=0$ for $c<i\le k$. The \textit{bihomogeneous Schur
    polynomial $s_\lambda$ in $k$ variables} is defined by
  \begin{align*}
    s_\lambda(x_1,y_1,x_2,y_2,\ldots,x_k,y_k)=\frac{a_{(\lambda_1+k-1,\lambda_2+k-2,\ldots,\lambda_k+0)}(x_1,y_1,\ldots,x_k,y_k)}{a_{(k-1,k-2,\ldots,0)}(x_1,y_1,\ldots,x_k,y_k)},
  \end{align*}
  where
  \begin{align*}
    a_{(l_1,l_2,\ldots,l_k)}(x_1,y_1,x_2,y_2,\ldots,x_k,y_k)=\det(x_j^{l_i}y_j^{l_1-l_i})
  \end{align*}
  is the Vandermonde determinant.
  Similarly, the
  \textit{bihomogeneous elementary symmetric polynomials} $e_j$ 
  are defined by
  \begin{align*}
    e_j(x_1,y_1,\ldots,x_k,y_k)&=\sum_{\substack{A\subseteq[k]\\\abs{A}=j}}\prod_{i\in
    A}x_i\prod_{i\not\in A}y_i
  \end{align*}
\end{Def}
Note that $s_\lambda$ is bihomogeneous of degree $\lambda_1$ and $e_j$
is bihomogeneous of degree 1 
in each pair of variables $x_i,y_i$. To avoid confusion, we point out
that $e_0=y_1\cdots y_k.$
\begin{notation}
  Schur polynomials in $k$ variables are indexed by partitions with at
  most $k$ parts, or alternatively by Young diagrams that fit inside a
  $k\times\infty$ rectangle. Since Young diagrams also appear in this
  paper in relation to monomial ideals, we distinguish them as
  follows. We draw Young diagrams related to Schur polynomials with
  the longest row on top (English notation), as opposed the way we
  have been drawing monomial ideals (French notation).
\end{notation}
We now give a correspondence between Young diagrams and sets of
monomials.
\begin{Def}\label{Def:RimPath}
  Fix $h,k\ge1.$ Let $\lambda$ be a partition whose Young diagram fits
  inside the $k\times h$ rectangle in $\Z^2$. (That is, $\lambda$ has
  at most $k$ parts, and $\lambda_1\le h.$) The \textit{width-$h$,
    height-$k$ rim path} of $\lambda$ is the lattice path $P_\lambda^{h,k}$
  in $\Z^2$ that begins at $(h,0)$, and follows the edge of the Young
  diagram down and to the left until it reaches $(0,-k).$ We index the
  steps of $P_\lambda^{h,k}$ by $i=0,\ldots,h+k-1.$ 

  The \textit{width-$h$, height-$k$ monomial set} of $\lambda$ is the
  set
  \begin{align*}
    U_\lambda^{h,k}=\{x^{h+k-1-i}y^i\in\Mon_{h+k-1}:i\in\{0,\ldots,h+k-1\}\text{
    such that the $i$th step
    of $P_\lambda^{h,k}$ is vertical}\}.
  \end{align*}
\end{Def}
The definition is illustrated on the left in Figure
\ref{fig:AllowedSegments}. Note that $\abs{U_\lambda^{h,k}}=k.$

\begin{rem}\label{rem:Bijection}
  The operation of taking the width-$h$, height-$k$ monomial set has a
  clear inverse, hence gives
  a bijection between partitions with at most $k$ parts and
  $\lambda_1\le h$, and $k$-element subsets of $\Mon_{h+k-1}.$ Thus
  Proposition \ref{prop:IntersectDependenceLoci} implies:
\end{rem}
\begin{prop}\label{prop:IntersectDependenceLociLambda}
  For any subset $U\subseteq\Mon_{h+k-1}$,
  \begin{align*}
    \mathcal{D}(U)=\bigcap_{\lambda:U_\lambda^{h,k}\supseteq U}\mathcal{D}(U_\lambda^{h,k}).
  \end{align*}
\end{prop}
We show how to visualize Proposition
\ref{prop:IntersectDependenceLociLambda} in an example.
\begin{ex}\label{ex:AdmissiblePartitions}
  Let $h=7$ and $k=5.$ Let $U=\{x^{11},x^6y^5,x^2y^9\}.$ If $\lambda$
  is such that $U\subseteq U_\lambda^{7,5}$, then the 0th, 5th, and
  9th steps of $P_\lambda^{7,5}$ are vertical. Concretely, this says
  precisely that the dashed red segments in Figure
  \ref{fig:AllowedSegments} are not in $P_\lambda^{7,5}.$ (This also
  disallows certain other segments from being in $P_\lambda^{7,5}$; we
  have shown these as dotted lines.) Then a partitions $\lambda$
  satisfies $U\subseteq U_\lambda^{7,5}$ if and only if
  $P_\lambda^{7,5}$ consists of solid segments. For example,
  $\lambda=(7,4,4,1)$ satisfies $U\subseteq U_\lambda^{7,5}$;
  $P_\lambda^{7,5}$ is drawn in bold in Figure
  \ref{fig:AllowedSegments}.
  
  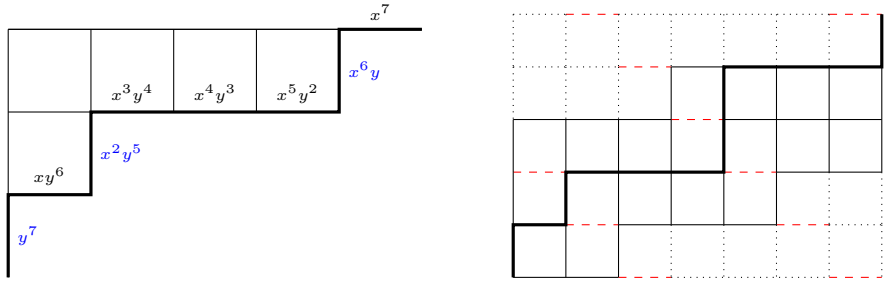
\begin{figure}[h]
    \centering
    \begin{tikzpicture}[scale=1.1]
      \draw (0,0)--(4,0);
      \draw (0,0)--(5,0);
      \draw (0,-1)--(4,-1);
      \draw (0,-2)--(1,-2);
      \draw (4,0)--(4,-1);
      \draw (3,0)--(3,-1);
      \draw (2,0)--(2,-1);
      \draw (1,0)--(1,-2);
      \draw (0,0)--(0,-3);
      \draw[very thick] 
      (0,-3)--(0,-2)--(1,-2)--(1,-1)--(4,-1)--(4,0)--(5,0);
      \draw (0.5,-2) node[above] {\tiny $xy^6$};
      \draw (1.5,-1) node[above] {\tiny $x^3y^4$};
      \draw (2.5,-1) node[above] {\tiny $x^4y^3$};
      \draw (3.5,-1) node[above] {\tiny $x^5y^2$};
      \draw (4.5,0) node[above] {\tiny $x^7$};
      \draw (4,-.5) node[right] {\tiny \textcolor{blue}{$x^6y$}};
      \draw (1,-1.5) node[right] {\tiny \textcolor{blue}{$x^2y^5$}};
      \draw (0,-2.5) node[right] {\tiny \textcolor{blue}{$y^7$}};
    \end{tikzpicture}
    \quad\quad\quad
    \begin{tikzpicture}[scale=0.7]
      \draw[dotted] (0,0)--(1,0); \draw[dashed,red] (1,0)--(2,0); \draw[dotted]
      (2,0)--(6,0); \draw[dashed,red] (6,0)--(7,0); \draw[dotted]
      (0,-1)--(2,-1); \draw[dashed,red] (2,-1)--(3,-1); \draw
      (3,-1)--(7,-1); \draw (0,-2)--(3,-2); \draw[dashed,red]
      (3,-2)--(4,-2); \draw (4,-2)--(7,-2); \draw[dashed,red]
      (0,-3)--(1,-3); \draw (1,-3)--(4,-3); \draw[dashed,red]
      (4,-3)--(5,-3); \draw (5,-3)--(7,-3); \draw (0,-4)--(1,-4);
      \draw[dashed,red] (1,-4)--(2,-4); \draw (2,-4)--(5,-4);
      \draw[dashed,red] (5,-4)--(6,-4); \draw[dotted] (6,-4)--(7,-4); \draw
      (0,-5)--(2,-5); \draw[dashed,red] (2,-5)--(3,-5); \draw[dotted]
      (3,-5)--(6,-5); \draw[dashed,red] (6,-5)--(7,-5);
      \foreach \i in {0,...,6} {
        \draw[dotted] (\i,0)--(\i,-1);
      }
      \foreach \i in {0,...,2} {
        \draw[dotted] (\i,-1)--(\i,-2);
      }
      \foreach \i in {6,...,7} {
        \draw[dotted] (\i,-3)--(\i,-4);
      }
      \foreach \i in {3,...,7} {
        \draw[dotted] (\i,-4)--(\i,-5);
      }
      \foreach \i in {7,...,7} {
        \draw (\i,0)--(\i,-1);
      }
      \foreach \i in {3,...,7} {
        \draw (\i,-1)--(\i,-2);
      }
      \foreach \i in {0,...,7} {
        \draw (\i,-2)--(\i,-3);
      }
      \foreach \i in {0,...,5} {
        \draw (\i,-3)--(\i,-4);
      }
      \foreach \i in {0,...,2} {
        \draw (\i,-4)--(\i,-5);
      }
      \draw[very thick]
      (7,0)--(7,-1)--(4,-1)--(4,-3)--(1,-3)--(1,-4)--(0,-4)--(0,-5);
    \end{tikzpicture}
    \caption{Left: The width-5, height-3 rim path of $\lambda=(4,1)$
      is drawn in bold, and the monomial set is
      $U_{(4,1)}^{5,3}=\{x^6y,x^2y^5,y^7\}$. Right: Allowed rim path
      segments from Example \ref{ex:AdmissiblePartitions}, with the
      example $\lambda=(7,4,4,1)$ in bold.}
    \label{fig:AllowedSegments}
  \end{figure}
\end{ex}
By Proposition \ref{prop:IntersectDependenceLociLambda}, the matroid
stratification of $(\P^1)^{[k]}$ is ``generated'' (via taking
intersections and complements) by the loci
$\mathcal{D}(U_\lambda^{h,k}).$ Thus the stratification is entirely
determined by the following:

\begin{thm}\label{thm:SchurPolynomials}
  The dependence subscheme
  $\mathcal{D}(U_\lambda^{h,k})\subseteq(\P^1)^{[k]}$ is the vanishing
  locus of
  $$e_0(x_1,\ldots,y_k)^{h-\lambda_1}s_\lambda(x_1,\ldots,y_k)$$ (via
  the isomorphism $(\P^1)^{[k]}\cong\Sym^k\P^1$).
\end{thm}
\begin{rem}
  A. Fink independently observed a connection between matroid strata
  in $(\P^1)^{[k]}$ and Schur polynomials.
\end{rem}
\begin{proof}
  For notational convenience, in this proof we will write $e_i$ for
  $e_i(x_1,\ldots,y_k)$ and $s_\lambda$ for
  $s_\lambda(x_1,y_1,\ldots,x_k,y_k)$.
  
  Recall the tautological sequences \eqref{eq:TautSeqD}. Let $f$ be a
  nonvanishing local section of the line bundle $\mathcal{I}_k.$ Then
  in terms of the roots $[x_1:y_1],\ldots,[x_k:y_k],$ we have
  $$f=\prod_{i=1}^k(y_ix-x_iy)=e_0x^k-e_1x^{k-1}y+\cdots+(-1)^ke_ky^k.$$
  
  \noindent \textbf{Step 0.} Since $e_0^{h-\lambda_1}s_\lambda$ is
  homogeneous in $y_1,\ldots,y_k$, we may instead work with the
  negative roots, and write
  \begin{align*}
    f=\prod_{i=1}^k(y_ix+x_iy)=e_0x^k+e_1x^{k-1}y+\cdots+e_ky^k.
  \end{align*}
  The lack of signs will simplify Step 2.

  \medskip
  
  \noindent\textbf{Step 1.} 
  By definition, $\mathcal{D}(U_\lambda^{h,k})$ is the vanishing locus
  of the section
  $$\sigma_{U_\lambda^{h,k}}\in\bigwedge\nolimits^k\mathcal{Q}_{h+k-1}=\Hom\left(\bigwedge\nolimits^k\Span(U_\lambda^{h,k}),\bigwedge\nolimits^k\mathcal{Q}_{h+k-1}\right)$$
  defined as the $k$th wedge of the chain of maps
  \begin{align*}
    \Span(U_\lambda^{h,k})\into\mathcal{R}_{h+k-1}\to\mathcal{Q}_{h+k-1}.
  \end{align*}
  By duality, there is a natural isomorphism
  $$\Hom\left(\bigwedge\nolimits^k\Span(U_\lambda^{h,k}),\bigwedge\nolimits^k\mathcal{Q}_{h+k-1}\right)\to\Hom\left(\left(\bigwedge\nolimits^k\mathcal{Q}_{h+k-1}\right)^{\vee},\left(\bigwedge\nolimits^k\Span(U_\lambda^{h,k})\right)^{\vee}\right).$$ The two exact sequences
  \begin{align*}
    0\to\mathcal{I}_{h+k-1}\to\mathcal{R}_{h+k-1}\to\mathcal{Q}_{h+k-1}\to0
  \end{align*}
  and
  \begin{align*}
    0\to\Span(U_\lambda^{h,k})\to\mathcal{R}_{h+k-1}\to\mathcal{Q}_{h+k-1}\to0
  \end{align*}
  give identifications
  \begin{align*}
  \left(\bigwedge\nolimits^k\mathcal{Q}_{h+k-1}\right)^{\vee}&\cong\bigwedge\nolimits^h\mathcal{I}_{h+k-1}\\\left(\bigwedge\nolimits^k\Span(U_\lambda^{h,k})\right)^{\vee}&\cong\bigwedge\nolimits^h\mathcal{R}_{h+k-1}/\Span(U_\lambda^{h,k}).
  \end{align*}
  Thus $\sigma_{U_\lambda^{h,k}}$ is identified with the section of
  $\Hom\left(\bigwedge^h\mathcal{I}_{h+k-1},\bigwedge^h\mathcal{R}_{h+k-1}/\Span(U_\lambda^{h,k})\right)$
  defined as the $h$-th (top) wedge of the chain of maps
  \begin{align}\label{eq:ABchain}
    \mathcal{I}_{h+k-1}\xrightarrow{A}{}\mathcal{R}_{h+k-1}\xrightarrow{B}{}\mathcal{R}_{h+k-1}/\Span(U_\lambda^{h,k}),
  \end{align}
  i.e. $\det(B\circ A).$ Note that
  \begin{align*}
    \Hom\left(\bigwedge\nolimits^h\mathcal{I}_{h+k-1},\bigwedge\nolimits^h\mathcal{R}_{h+k-1}/\Span(U_\lambda^{h,k})\right)&\cong\Hom\left(\bigwedge\nolimits^h(\mathcal{R}_{h-1}\otimes\mathcal{I}_k),\bigwedge\nolimits^h\mathcal{R}_{h+k-1}/\Span(U_\lambda^{h,k})\right)\\
                                                                                         &\cong\Hom\left(\bigwedge\nolimits^h\mathcal{R}_{h-1}\otimes\mathcal{I}_k^{\otimes h},\bigwedge\nolimits^h\mathcal{R}_{h+k-1}/\Span(U_\lambda^{h,k})\right)\\
                                                                                         &\cong(\mathcal{I}_k^*)^{\otimes h}\cong\O(h),
  \end{align*}
  
  \noindent\textbf{Step 2.} By principality, 
  there is a natural multiplication isomorphism
  $$R_{h-1}\otimes\mathcal{I}_k\to\mathcal{I}_{h+k-1}.$$
  The inclusion $A$ from \eqref{eq:ABchain} has the following matrix
  $X_A$ with respect to the basis
  $\{m\otimes f:\text{monomials }m\in R_{h-1}\}$ for
  $\mathcal{I}_{h+k-1}$ and the monomial basis for $R_{h+k-1}$
  (ordered such that larger powers of $x$ appear first):
  \begin{align}\label{eq:ElemMatrix}
    X_A=\left(e_{b-j}\right)_{0\le b\le h+k-1}^{0\le j\le h-1}=
    \begin{pmatrix}
      e_0&0&\cdots&0&0\\
      e_1&e_0&\cdots&0&0\\
      e_2&e_1&\ddots&e_0&0\\
      \vdots&\vdots&\ddots&e_1&e_0\\
      e_k&e_{k-1}&\ddots&e_2&e_1\\
      0&e_k&\ddots&\vdots&e_2\\
      \vdots&\vdots&\ddots&e_k&\vdots\\
      0&0&\cdots&0&e_k
    \end{pmatrix}
  \end{align}
  The (square) matrix $X_{B\circ A}$ of $B\circ A$ is obtained by
  deleting the rows corresponding to elements of
  $U_\lambda^{h,k}$. Let $X_{B\circ A}'$ be the matrix obtained by
  reversing the order of the rows and the order of the columns in
  $X_{B\circ A}$. Define $b_0,\ldots,b_{h-i}$ so that the $i$-th row
  of $X_{B\circ A}'$ is the $b_i$-th row of $X_A.$




  Note that rows of $X_{B\circ A}'$ correspond to rightward steps in
  the \textit{reversed} width-$h$, height-$k$ rim path of $\lambda$
  --- that is, to columns in the Young diagram of $\lambda$. The
  $i$-th row of $X_{B\circ A}'$ (starting with $i=0$) has entries
  $e_{\ell_i},e_{\ell_i+1},\ldots,e_{\ell_i+h-1},$ where
  $$\ell_i=k-i-\#\{b\le b_i:x^{b}y^{h+k-1-b}\in
  U_\lambda^{h,k}\}.$$ Since elements of
  $\{b\le b_i:x^{b}y^{h+k-1-b}\in U_\lambda^{h,k}\}$ correspond to
  upward steps in the reversed rim path, we see that
  $\ell_i+i=k-\#\{b\le b_i:x^{b}y^{h+k-1-b}\}$ is the $i$-th entry of
  the conjugate partition $\lambda'$. (Here $\lambda'$ is taken to
  have exactly $h$ entries, some of which may be zero.) Thus
  $X_{B\circ A}'=(e_{\lambda'_i+j-i})_{i,j=0}^{h-1}$. Note that
  $\ell_i+i=0$ for $\lambda_1<i\le h.$ Expanding the determinant along
  the last $h-\lambda_1$ rows gives
  $$\det(X_{B\circ A})=e_0^{h-\lambda_1}\det((e_{\lambda'_i+j-i})_{i,j=0}^{\lambda_1}).$$
  By the second Jacobi-Trudi formula,
  $\det(X_{B\circ A})=\pm\det(X_{B\circ A}')=\pm e_0^{h-\lambda_1}s_\lambda.$
\end{proof}

\begin{rem}\label{rem:FactorOutMoreElems}
  If $\lambda_k>0,$ 
  then expanding the Jacobi-Trudi formula gives
  \begin{align*}
    s_\lambda=e_k^{\lambda_k}s_{(\lambda_1-\lambda_k,\lambda_2-\lambda_k,\ldots,\lambda_{k-1}-\lambda_k)}(x_1,\ldots,y_k).
  \end{align*}
  In particular, Theorem \ref{thm:SchurPolynomials} now implies that
  \begin{align*}
    \mathcal{D}(U_\lambda^{h,k})=V(e_0^{h-\lambda_1}e_k^{\lambda_k}s_{(\lambda_1-\lambda_k,\lambda_2-\lambda_k,\ldots,\lambda_{k-1}-\lambda_k)}(x_1,\ldots,y_k)).
  \end{align*}
  This is a manifestation of Proposition \ref{prop:StrataPullback}.
\end{rem}
\begin{rem}\label{rem:IntersectionsOfSchurDivisors}
  Theorem \ref{thm:SchurPolynomials} reduces all questions about the
  matroid stratification to questions about the intersection theory
  of Schur polynomials -- however, it appears that intersection theory
  of Schur polynomials has not been actively studied.
\end{rem}
\begin{rem}
  If $\mathbbm{k}=\C$ (or more generally if $\mathbbm{k}$ is
  uncountable and algebraically closed), it follows from Theorem
  \ref{thm:SchurPolynomials} that for a very general point
  $I\in(\P^1)^{[k]}$, $\Trop(I_d)$ is the uniform matroid of the
  appropriate rank, for all $d$. This is a special case of a
  forthcoming result by Maclagan and the author, which states that
  under the same assumptions on $\mathbbm{k},$ a very general point
  $I\in\Hilb_{(a,b)}^h(\A^2)$ satisfies
  $\Trop(I_d)=U_{h(d),\Mon_d(a,b)}$ for all $d$, where $(a,b)$ is any
  positive grading and $h$ is any Hilbert function. (Recall Example
  \ref{ex:UniformMatroid}.)
\end{rem}

\begin{ex}\label{ex:Necklace2}
  We continue Example \ref{ex:Necklace1}. Fix $k\ge1$ and $d\ge k.$ By
  Proposition \ref{prop:IntersectDependenceLociLambda} and Theorem
  \ref{thm:SchurPolynomials}, there is a certain set of Schur
  polynomials in $k$ variables whose common vanishing locus is a
  collection of rational curves indexed by the set $N_{d,k}$ of binary
  necklaces with $k$ black and $d-k$ white beads. One may also show
  this directly, as follows.

  By the analysis in Example \ref{ex:AdmissiblePartitions}, the Schur
  polynomials in question are precisely those $s_\lambda$ such that
  $\lambda$ has at most $k-1$ parts, and $\lambda_1=d-k+1.$ The
  vanishing of these polynomials is highly non-transverse; however,
  calculations using the first Jacobi-Trudi formula show that the
  ideal they generate is in fact equal to the ideal
  $J=(h_{d-k+1},h_{d-k+2},\ldots,h_{d-1})$, where $h_i$ is the $i$-th
  (bihomogeneous) \textit{complete} symmetric polynomial,
  i.e.
  $$h_i=\sum_{\substack{i_1,\ldots,i_k\ge0\\i_1+\cdots+i_k=i}}\thickspace\thickspace\prod_{j=1}^kx_j^{i_j}y_j^{i-i_j}.$$ By
  \cite{ConcaKrattenthalerWatanabe2009}, these polynomials form a
  regular sequence, so $V(J)\subseteq(\P^1)^{[k]}$ is
  1-dimensional, as desired.

  One may show directly that if $z_1,\ldots,z_k\in\C\P^1$ are distinct
  vertices of a regular $d$-gon centered at 0, then the polynomials
  $h_{d-k+1},\ldots,h_{d-1}$ vanish at $(z_1,\ldots,z_k).$ This shows
  that $\mathcal{D}(U)$ contains the collection of rational curves in
  Example \ref{ex:Necklace1}. One may then show by a degree
  calculation that the $\mathcal{D}(U)$ does not contain any other
  points.
\end{ex}

\subsection{An open problem interlude: The tropical ideal associated
  to a necklace}\label{sec:NecklaceToMatroid}
Following Examples \ref{ex:Necklace1} and \ref{ex:Necklace2}, we now
pose a natural combinatorial question, to which we do not know the
answer. Let $\gamma\in N_{d,k}$ be a necklace with $k$ black beads and
$d-k$ white beads. There is a corresponding curve $C_\gamma\cong\C^*$
in $\mathcal{D}(\{x^d,y^d\})\subseteq(\P^1)^{[k]}$. In fact, as
$C_\gamma$ is a torus orbit (see Section \ref{sec:Orbits}), it has the
property that any $I\in C_\gamma$ has the same tropicalization
$\Trop(\gamma):=\Trop(I).$ (In other words, $C_\gamma$ is in a single
matroid stratum; we will see that it may not be an entire matroid
stratum.) For example, see Figure \ref{fig:D62} and
Figure \ref{fig:N84}.
\begin{question}\label{q:MGammaAlgorithm}
  Is there a combinatorial algorithm to compute the function
  $\gamma\mapsto\Trop(\gamma)$?
\end{question}
\begin{figure}
  \centering
  \begin{align*}
    \begin{array}{ccc}
      \includegraphics[height=1.1in]{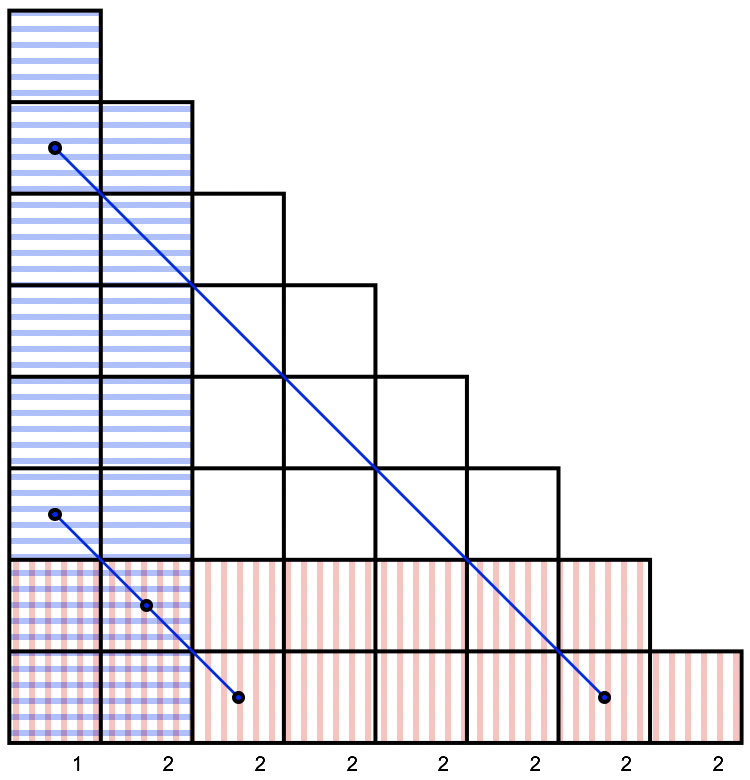}&
                                            \includegraphics[height=1.1in]{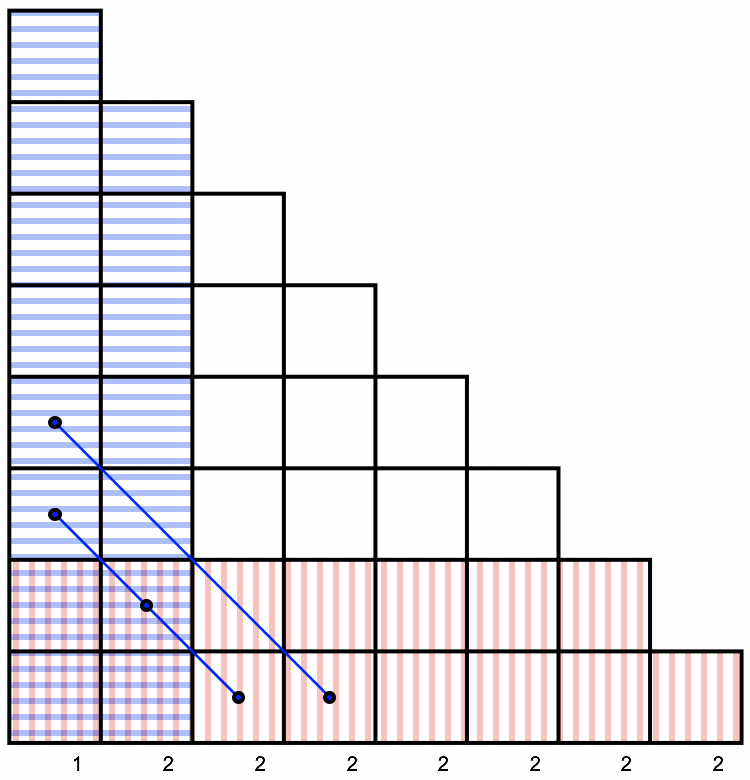}&
                                                                                  \includegraphics[height=1.1in]{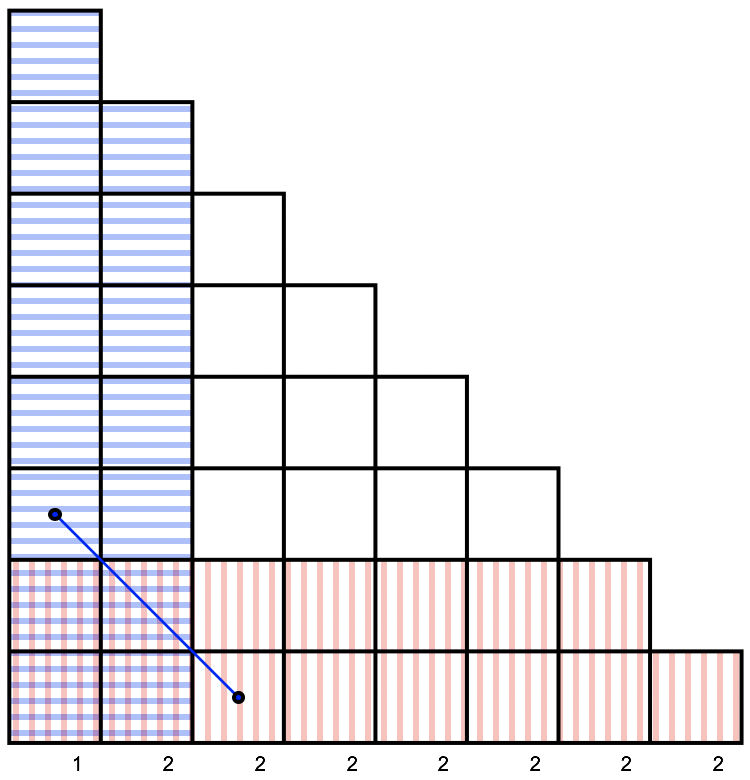}\\
      x_2^2y_1^2-x_1x_2y_2y_1+x_1^2y_2^2=0&x_2^2y_1^2+x_1x_2y_2y_1+x_1^2y_2^2=0&x_2y_1+x_1y_2=0\\
      \begin{tikzpicture}[scale=.7]
        \draw
        (0:1)--(60:1)--(120:1)--(180:1)--(240:1)--(300:1)--cycle;
        \filldraw[fill=black] (0:1) circle(.1);
        \filldraw[fill=black] (60:1) circle(.1);
        \filldraw[fill=white] (120:1) circle(.1);
        \filldraw[fill=white] (180:1) circle(.1);
        \filldraw[fill=white] (240:1) circle(.1);
        \filldraw[fill=white] (300:1) circle(.1);
        \draw (0,1) node{};
      \end{tikzpicture}
        &\begin{tikzpicture}[scale=.7]
        \draw
        (0:1)--(60:1)--(120:1)--(180:1)--(240:1)--(300:1)--cycle;
        \filldraw[fill=black] (0:1) circle(.1);
        \filldraw[fill=white] (60:1) circle(.1);
        \filldraw[fill=black] (120:1) circle(.1);
        \filldraw[fill=white] (180:1) circle(.1);
        \filldraw[fill=white] (240:1) circle(.1);
        \filldraw[fill=white] (300:1) circle(.1);
        \draw (0,1) node{};
      \end{tikzpicture}
          &\begin{tikzpicture}[scale=.7]
        \draw
        (0:1)--(60:1)--(120:1)--(180:1)--(240:1)--(300:1)--cycle;
        \filldraw[fill=black] (0:1) circle(.1);
        \filldraw[fill=white] (60:1) circle(.1);
        \filldraw[fill=white] (120:1) circle(.1);
        \filldraw[fill=black] (180:1) circle(.1);
        \filldraw[fill=white] (240:1) circle(.1);
        \filldraw[fill=white] (300:1) circle(.1);
        \draw (0,1) node{};
      \end{tikzpicture}
    \end{array}
  \end{align*}  
  \caption{The three elements of $N_{6,2}$ and their
    tropicalizations. Each equation defines (the closure of) the
    corresponding stratum in $(\P^1)^{[2]}\cong\Sym^2\P^1$.}
  \label{fig:D62}
\end{figure}
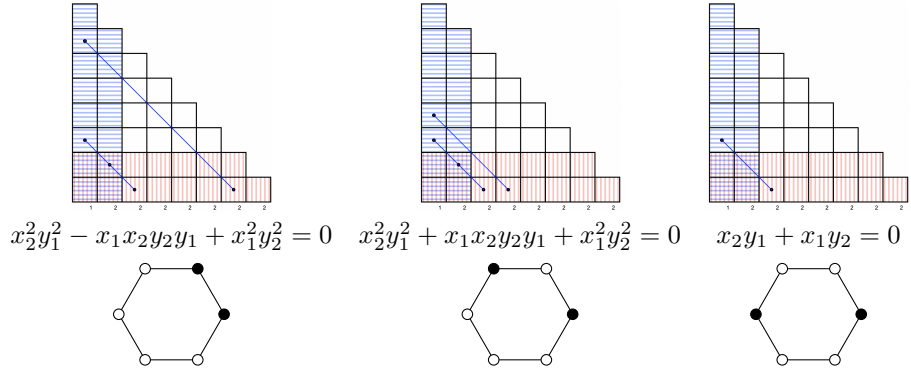
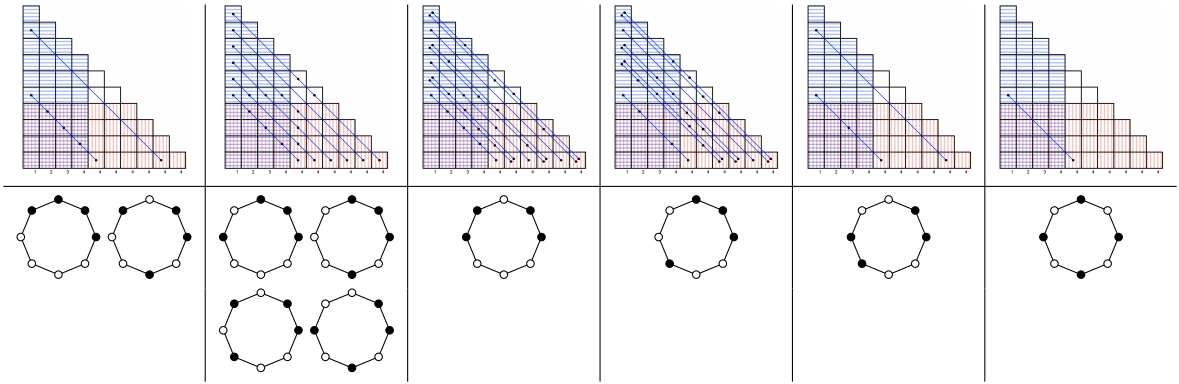
\begin{figure}
  \centering
  \begin{tabular}{c|c|c|c|c|c}
    \includegraphics[height=.9in]{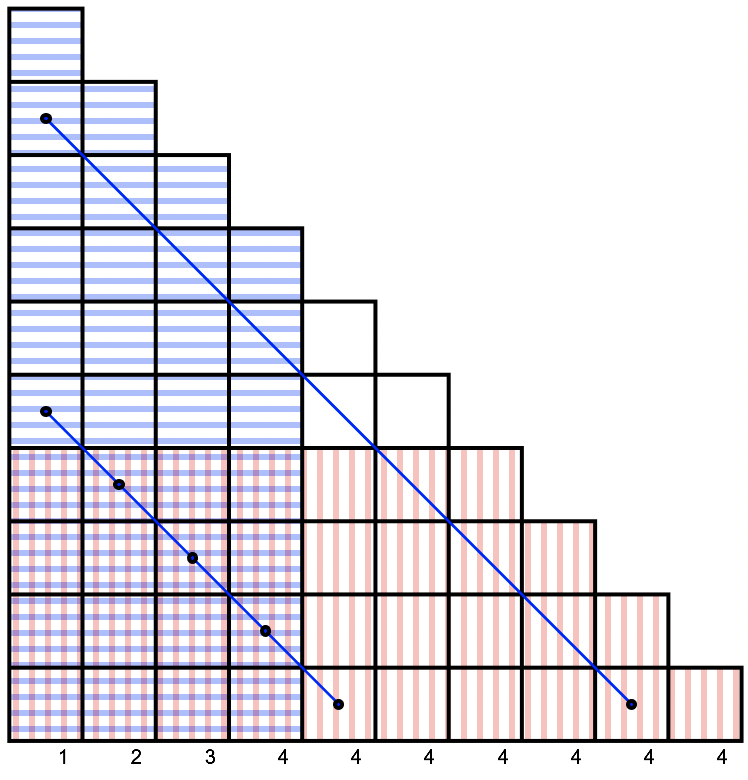}&\includegraphics[height=.9in]{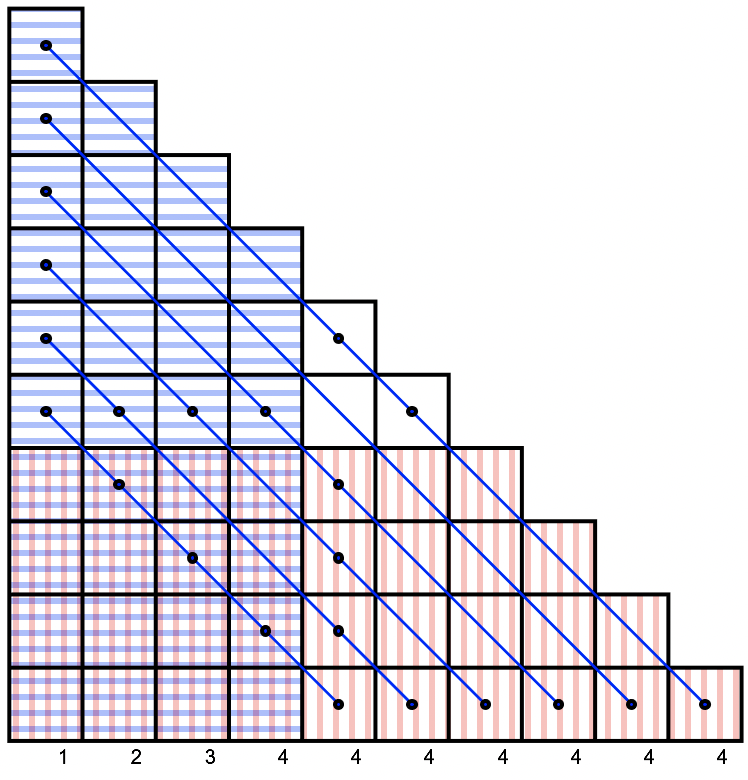}&\includegraphics[height=.9in]{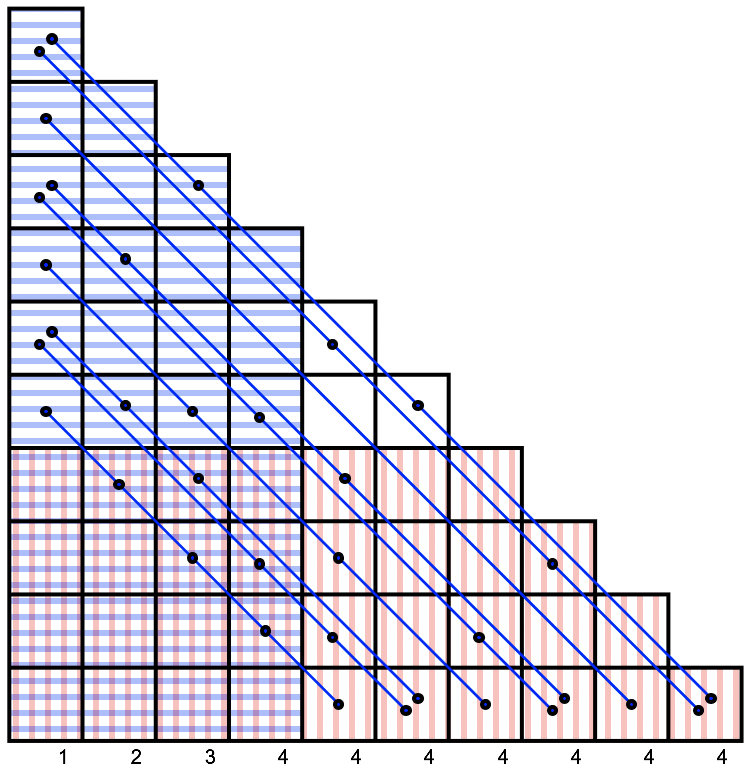}&\includegraphics[height=.9in]{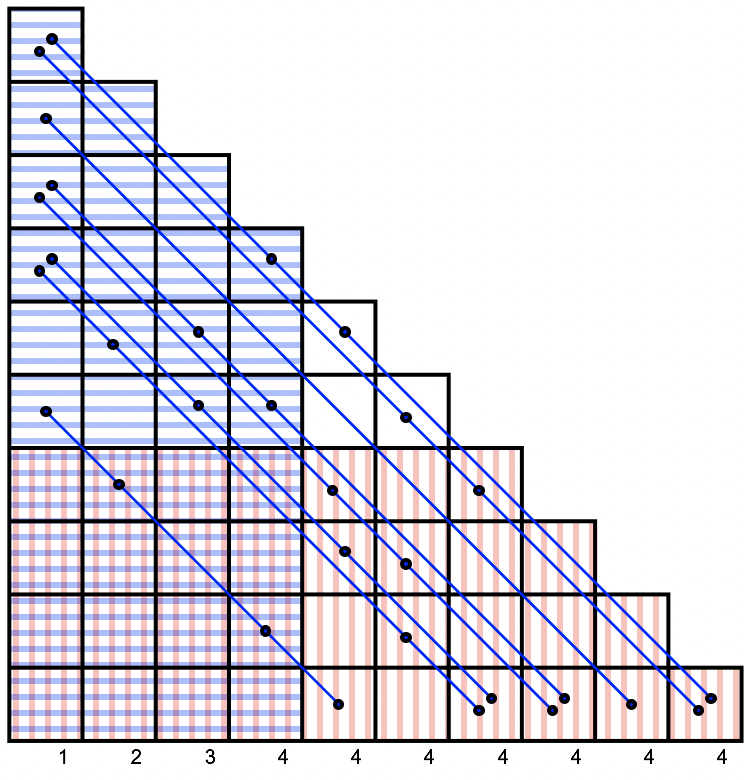}&\includegraphics[height=.9in]{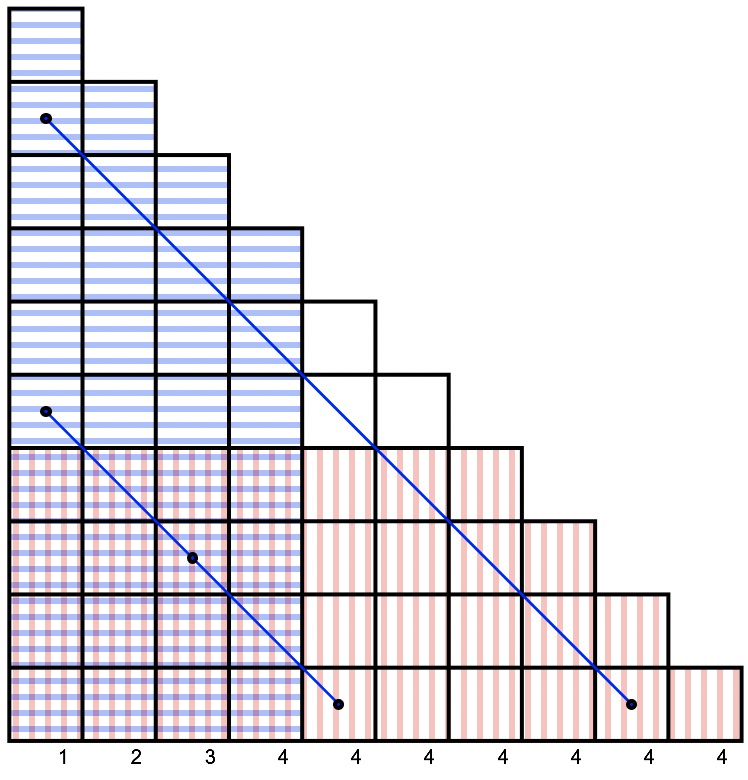}&\includegraphics[height=.9in]{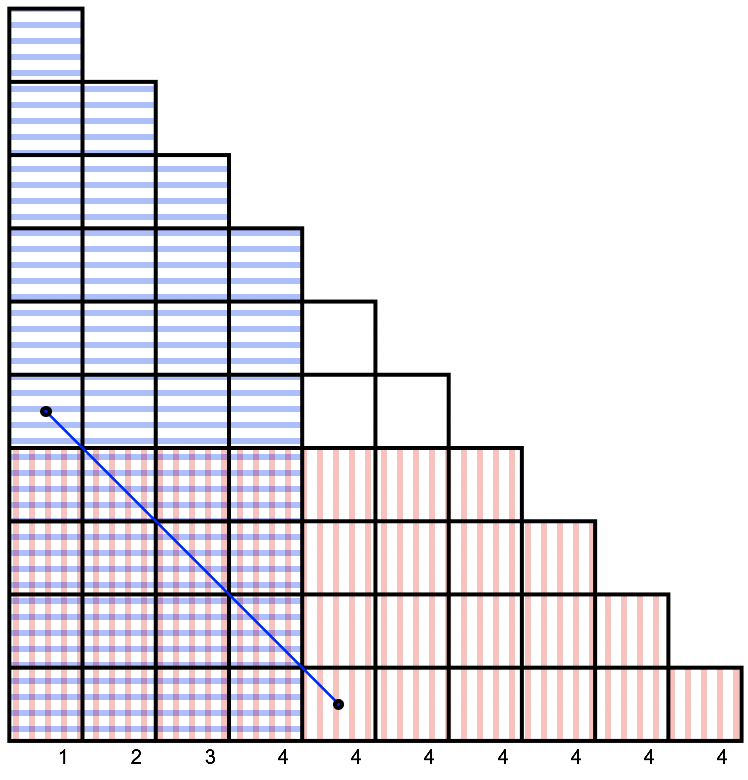}\\\hline
    \begin{tikzpicture}[scale=0.5]
      \draw
      (0:1)--(45:1)--(90:1)--(135:1)--(180:1)--(225:1)--(270:1)--(315:1)--cycle;
      \filldraw[fill=black] (0:1) circle(.1);
      \filldraw[fill=black] (45:1) circle(.1);
      \filldraw[fill=black] (90:1) circle(.1);
      \filldraw[fill=black] (135:1) circle(.1);
      \filldraw[fill=white] (180:1) circle(.1);
      \filldraw[fill=white] (225:1) circle(.1);
      \filldraw[fill=white] (270:1) circle(.1);
      \filldraw[fill=white] (315:1) circle(.1);
      \draw (90:1.1) node {};
    \end{tikzpicture}\thickspace\begin{tikzpicture}[scale=0.5]
      \draw
      (0:1)--(45:1)--(90:1)--(135:1)--(180:1)--(225:1)--(270:1)--(315:1)--cycle;
      \filldraw[fill=black] (0:1) circle(.1);
      \filldraw[fill=black] (45:1) circle(.1);
      \filldraw[fill=white] (90:1) circle(.1);
      \filldraw[fill=black] (135:1) circle(.1);
      \filldraw[fill=white] (180:1) circle(.1);
      \filldraw[fill=white] (225:1) circle(.1);
      \filldraw[fill=black] (270:1) circle(.1);
      \filldraw[fill=white] (315:1) circle(.1);
    \end{tikzpicture}&\begin{tikzpicture}[scale=0.5]
      \draw
      (0:1)--(45:1)--(90:1)--(135:1)--(180:1)--(225:1)--(270:1)--(315:1)--cycle;
      \filldraw[fill=black] (0:1) circle(.1);
      \filldraw[fill=black] (45:1) circle(.1);
      \filldraw[fill=black] (90:1) circle(.1);
      \filldraw[fill=white] (135:1) circle(.1);
      \filldraw[fill=black] (180:1) circle(.1);
      \filldraw[fill=white] (225:1) circle(.1);
      \filldraw[fill=white] (270:1) circle(.1);
      \filldraw[fill=white] (315:1) circle(.1);
    \end{tikzpicture}\thickspace\begin{tikzpicture}[scale=0.5]
      \draw
      (0:1)--(45:1)--(90:1)--(135:1)--(180:1)--(225:1)--(270:1)--(315:1)--cycle;
      \filldraw[fill=black] (0:1) circle(.1);
      \filldraw[fill=black] (45:1) circle(.1);
      \filldraw[fill=black] (90:1) circle(.1);
      \filldraw[fill=white] (135:1) circle(.1);
      \filldraw[fill=white] (180:1) circle(.1);
      \filldraw[fill=white] (225:1) circle(.1);
      \filldraw[fill=black] (270:1) circle(.1);
      \filldraw[fill=white] (315:1) circle(.1);
    \end{tikzpicture}&\begin{tikzpicture}[scale=0.5]
      \draw
      (0:1)--(45:1)--(90:1)--(135:1)--(180:1)--(225:1)--(270:1)--(315:1)--cycle;
      \filldraw[fill=black] (0:1) circle(.1);
      \filldraw[fill=black] (45:1) circle(.1);
      \filldraw[fill=white] (90:1) circle(.1);
      \filldraw[fill=black] (135:1) circle(.1);
      \filldraw[fill=black] (180:1) circle(.1);
      \filldraw[fill=white] (225:1) circle(.1);
      \filldraw[fill=white] (270:1) circle(.1);
      \filldraw[fill=white] (315:1) circle(.1);
    \end{tikzpicture}&\begin{tikzpicture}[scale=0.5]
      \draw
      (0:1)--(45:1)--(90:1)--(135:1)--(180:1)--(225:1)--(270:1)--(315:1)--cycle;
      \filldraw[fill=black] (0:1) circle(.1);
      \filldraw[fill=black] (45:1) circle(.1);
      \filldraw[fill=black] (90:1) circle(.1);
      \filldraw[fill=white] (135:1) circle(.1);
      \filldraw[fill=white] (180:1) circle(.1);
      \filldraw[fill=black] (225:1) circle(.1);
      \filldraw[fill=white] (270:1) circle(.1);
      \filldraw[fill=white] (315:1) circle(.1);
    \end{tikzpicture}&\begin{tikzpicture}[scale=0.5]
      \draw
      (0:1)--(45:1)--(90:1)--(135:1)--(180:1)--(225:1)--(270:1)--(315:1)--cycle;
      \filldraw[fill=black] (0:1) circle(.1);
      \filldraw[fill=black] (45:1) circle(.1);
      \filldraw[fill=white] (90:1) circle(.1);
      \filldraw[fill=white] (135:1) circle(.1);
      \filldraw[fill=black] (180:1) circle(.1);
      \filldraw[fill=black] (225:1) circle(.1);
      \filldraw[fill=white] (270:1) circle(.1);
      \filldraw[fill=white] (315:1) circle(.1);
    \end{tikzpicture}&\begin{tikzpicture}[scale=0.5]
      \draw
      (0:1)--(45:1)--(90:1)--(135:1)--(180:1)--(225:1)--(270:1)--(315:1)--cycle;
      \filldraw[fill=black] (0:1) circle(.1);
      \filldraw[fill=white] (45:1) circle(.1);
      \filldraw[fill=black] (90:1) circle(.1);
      \filldraw[fill=white] (135:1) circle(.1);
      \filldraw[fill=black] (180:1) circle(.1);
      \filldraw[fill=white] (225:1) circle(.1);
      \filldraw[fill=black] (270:1) circle(.1);
      \filldraw[fill=white] (315:1) circle(.1);
    \end{tikzpicture}
    \\
                                        &\begin{tikzpicture}[scale=0.5]
                                          \draw
                                          (0:1)--(45:1)--(90:1)--(135:1)--(180:1)--(225:1)--(270:1)--(315:1)--cycle;
                                          \filldraw[fill=black] (0:1) circle(.1);
                                          \filldraw[fill=black] (45:1) circle(.1);
                                          \filldraw[fill=white] (90:1) circle(.1);
                                          \filldraw[fill=black] (135:1) circle(.1);
                                          \filldraw[fill=white] (180:1) circle(.1);
                                          \filldraw[fill=black] (225:1) circle(.1);
                                          \filldraw[fill=white] (270:1) circle(.1);
                                          \filldraw[fill=white] (315:1) circle(.1);
                                        \end{tikzpicture}\thickspace\begin{tikzpicture}[scale=0.5]
                                          \draw
                                          (0:1)--(45:1)--(90:1)--(135:1)--(180:1)--(225:1)--(270:1)--(315:1)--cycle;
                                          \filldraw[fill=black] (0:1) circle(.1);
                                          \filldraw[fill=black] (45:1) circle(.1);
                                          \filldraw[fill=white] (90:1) circle(.1);
                                          \filldraw[fill=white] (135:1) circle(.1);
                                          \filldraw[fill=black] (180:1) circle(.1);
                                          \filldraw[fill=white] (225:1) circle(.1);
                                          \filldraw[fill=black] (270:1) circle(.1);
                                          \filldraw[fill=white] (315:1) circle(.1);
                                        \end{tikzpicture}&&&&
  \end{tabular}
  
  \caption{The ten elements of $N_{8,4}$ and their tropicalizations.}
  \label{fig:N84}
\end{figure}

We do not have a full answer to this question, but we now discuss it
further. Let $\gamma\in N_{d,k}.$ We know that $\{x^d,y^d\}$ is
dependent in $\Trop(\gamma).$ Note that for any $d'\ge k,$
$\{x^{d'},y^{d'}\}$ is dependent in $\Trop(\gamma)$ if and only if the
black beads of $\gamma$ are a subset of the vertices of a regular
$d'$-gon. Rephrasing this:
\begin{prop}
  Let $\gamma\in N_{d,k},$ and let $\alpha$ be the $\gcd$ of the $k$
  distances between consecutive beads in $\gamma$. (Since the sum of
  these distances is $d$, $d$ is divisible by $\alpha$.) Then
  $\{x^{d'},y^{d'}\}$ is dependent in $\Trop(\gamma)$ if and only if
  $d'$ is a multiple of $d/\alpha.$
\end{prop}
Note that this explains all circuits in Figure \ref{fig:D62}.
We also note the following condition, which implies certain necklaces
have the same tropicalization. 
\begin{prop}\label{prop:NecklaceSkip}
  Let $\gamma\in N_{d,k},$ and let $a\in(\Z/d\Z)^\times$. We define
  $a\gamma$ to be the necklace obtained by traversing $\gamma$ by
  jumps of length $a$. (For example, if $\gamma=\raisebox{-2pt}{
  \begin{tikzpicture}[scale=.2]
    \draw (0:1)--(72:1)--(144:1)--(216:1)--(288:1)--(0:1);
    \draw[red] (0:1)--(72:1);
    \filldraw[fill=black] (0:1) circle(.2);
    \filldraw[fill=black] (72:1) circle(.2);
    \filldraw[fill=white] (144:1) circle(.2);
    \filldraw[fill=white] (216:1) circle(.2);
    \filldraw[fill=white] (288:1) circle(.2);
    \draw (0,1) node{};
  \end{tikzpicture}
  }
  $, then $3\gamma=\raisebox{-2pt}{
  \begin{tikzpicture}[scale=.2]
    \draw (216:1)--(72:1)--(288:1)--(144:1)--(0:1);
    \draw[red] (0:1)--(216:1);
    \filldraw[fill=black] (0:1) circle(.2);
    \filldraw[fill=black] (72:1) circle(.2);
    \filldraw[fill=white] (144:1) circle(.2);
    \filldraw[fill=white] (216:1) circle(.2);
    \filldraw[fill=white] (288:1) circle(.2);
    \draw (0,1) node{};
  \end{tikzpicture}
  }=\raisebox{-2pt}{
  \begin{tikzpicture}[scale=.2]
    \draw (0:1)--(72:1)--(144:1)--(216:1)--(288:1)--(0:1);
    \draw[red] (0:1)--(72:1);
    \filldraw[fill=black] (0:1) circle(.2);
    \filldraw[fill=white] (72:1) circle(.2);
    \filldraw[fill=black] (144:1) circle(.2);
    \filldraw[fill=white] (216:1) circle(.2);
    \filldraw[fill=white] (288:1) circle(.2);
    \draw (0,1) node{};
  \end{tikzpicture}
  }
  $). Then $\Trop(\gamma)=\Trop(a\gamma).$
\end{prop}
\begin{proof}
  The independence of any $k$-element set $U_\lambda^{g,k}$ in
  $\Trop(\gamma)$ is determined by the nonvanishing of an element of
  $\C$ obtained by field operations applied to a primitive $d$th root
  of unity $\zeta$ (namely, the determinant of the associated Schur
  matrix). This nonvanishing is preserved by the field automorphism
  that sends $\zeta\mapsto\zeta^a,$ which determines the independence
  of $U_\lambda^{g,k}$ in $\Trop(\gamma)$.
\end{proof}
\begin{question}
  Does the converse of Proposition \ref{prop:NecklaceSkip} hold? That
  is, can there exist $\gamma_1,\gamma_2\in N_{d,k}$ such that
  $\Trop(\gamma_1)=\Trop(\gamma_2),$ but $\gamma_2\ne a\gamma_1$ for
  $a\in(\Z/d\Z)^\times$?  Observe that no counterexamples appear in
  Figures \ref{fig:D62} or \ref{fig:N84}.
\end{question}

In order to fully characterize $\Trop(\gamma)$, we need to know not
only which sets $\{x^{d'},y^{d'}\}$ are dependent, but which sets
$U_\lambda^{g,k}$ are dependent.

Let $\gamma\in N_{d,k}.$ Given $\lambda$ a partition with at most $k$
parts, let
$\eta_{d,k}(\lambda)=\{\zeta_d^{\lambda_i+k-i-1}:i=1,\ldots,k\}$. Let
$\zeta_d^{a_1},\ldots,\zeta_d^{a_k}$ be a set of points representing
$\gamma,$ and note that
\begin{align*}
  s_\lambda(\zeta_d^{a_1},\ldots,\zeta_d^{a_k})=\det\left((\zeta_d^{a_j(\lambda_i-i-1+k)})_{i,j=1}^k\right)/V,
\end{align*}
where $V$ is a Vandermonde determinant (which is guaranteed to be
nonzero at $\zeta_d^{a_1},\ldots,\zeta_d^{a_k}$). If $\lambda$ is such
that two elements of $\eta_{d,k}(\lambda)$ coincide, then
$s_\lambda(\zeta_d^{a_1},\ldots,\zeta_d^{a_k})=0$ since two rows of
the defining matrix are equal.

On the other hand, if $\eta_{d,k}(\lambda)$ contains $k$ distinct
elements, then $\eta_{d,k}(\lambda)$ naturally corresponds to a
necklace with $k$ black beads and $d-k$ white beads. In particular,
reordering and scaling $\eta_{d,k}(\lambda)$ corresponds to reordering
and scaling the rows of the matrix in the definition of
$s_\lambda(\zeta_d^{a_1},\ldots,\zeta_d^{a_k})$, which does not affect
its rank -- hence, the question of whether
$\gamma\in\mathcal{D}(U_{\lambda}^{g,k})$ for some $\gamma\in N_{d,k}$
depends only on the necklace $\eta_{d,k}(\lambda),$ not $\lambda$
itself.  This dependence is, interestingly, commutative in the
following sense.
\begin{prop}\label{prop:CommutativeNecklaces}
  Let $\gamma\in N_{d,k}$ such that $\gamma=\gamma(\lambda)$. Then
  $\gamma\in\mathcal{D}(U_{\lambda'}^{g,k})$ if and only if
  $\gamma(\lambda')\in\mathcal{D}(U_\lambda^{g,k}).$
\end{prop}
\begin{proof}
  This follows immediately from $\det(A)=\det(A^T).$
\end{proof}
Answering Question \ref{q:MGammaAlgorithm} now boils down to:
\begin{question}
  Let $\gamma_1$ and $\gamma_2$ be necklaces. We choose
  identifications of $\gamma_1$ and $\gamma_2$ with $k$-element
  subsets $\{\gamma_{1,i}\}$ and $\{\gamma_{2,j}\}$ of $\Z/d\Z$. Is
  there a combinatorial algorithm to determine whether the determinant
  $D(\gamma_1,\gamma_2):=\det\left((\zeta_d^{\gamma_{1,i}\gamma_{2,j}})_{i,j=1}^k\right)$
  vanishes?
\end{question}
\begin{rem}
  Experimentally, one may find sufficient conditions for the vanishing
  of the above determinant. In particular, one may prove a statement
  of the following form: if $a$ divides $d$, and the $k$ black beads
  of $\gamma_1$ are distributed ``sufficiently unequally'' among the
  $\mu_a$-orbits of the $d$th roots of unity, \textit{and} the $k$
  black beads of $\gamma_2$ are distributed ``sufficiently unequally''
  among the $\mu_{d/a}$-orbits of the $d$th roots of unity, then
  $D(\gamma_1,\gamma_2)=0.$ However, we do not know of any necessary
  conditions; an additional idea would be needed to prove that any
  determinants are \textit{nonzero}.
\end{rem}

\section{Equivariant structure of $\Hilb_{\mathbf{a}}^h(\A^r)$ and the
  $T$-graph problem}\label{sec:Orbits}
In this section, we rephrase some of the preceding material in terms
of torus actions. Let $\mathbbm{k}$ be algebraically closed in this
section. The condition of $\mathbf{a}$-homogeneity for an ideal $I$ is
equivalent to the invariance of $I$ under the action of a certain
subtorus of $T:=(\mathbbm{k}^*)^r$ --- specifically, the image of the
homomorphism $(\mathbbm{k}^*)^b\to(\mathbbm{k}^*)^r$ defined by the
matrix of exponents $\mathbf{a}$.
\begin{ex} The ideal
  $(x^2y+z^6)\subseteq \mathbbm{k}[x,y,z]$ is homogeneous with respect
  to the grading $((3,0),(0,6),(1,1))$. This ideal is invariant under
  elements of $T=(\mathbbm{k}^*)^3$ of the form
  $(\lambda_1^3,\lambda_2^6,\lambda_1\lambda_2),$ which act on the
  polynomial $x^2y+z^6$ by multiplication by $\lambda_1^6\lambda_2^6.$
\end{ex}
The torus $T$ acts on $\Hilb^h_{\mathbf{a}}(\A^r)$ with stabilizers
isomorphic to $T/(\mathbbm{k}^*)^b$, so the dimension of any $T$-orbit
$T\cdot I$ is at most $r-b$ (by nondegeneracy). There is a
stratification of $\Hilb^h_{\mathbf{a}}(\A^r)$ by $T$-orbit dimension;
it is easy to check that the (finite set of) monomial ideals with
graded Hilbert function $h$ are the $T$-fixed points.

If $b=r-1,$ then for every $h$ and $\mathbf{a},$ every $T$-orbit in
$\Hilb^h_{\mathbf{a}}(\A^r)$ is either a monomial ideal or is
isomorphic to $T/(\mathbbm{k}^*)^b\cong\mathbbm{k}^*.$ Each
1-dimensional orbit $T\cdot I$ has two ``endpoints;'' these are
initial monomial ideals of $I$ with respect to appropriate term
orders. (When $r=2,$ the two term orders are $x>y$ and $y>x.$ See
\cite{HeringMaclagan2012}.)
\begin{obs}
  Since $\Trop(I)$ is defined in terms of supports of polynomials, and
  $T$ acts by multiplying coefficients by nonzero scalars, we always
  have $\Trop(T\cdot I)=\Trop(I)$. In particular, every dependence
  locus and stratum of the matroid stratification is
  $T$-invariant. As initial ideals $\init_{\prec}(I)$ are defined via
  supports of polynomials, they are recoverable from
  $\Trop(I)$, as in the following definition.
\end{obs}
\begin{Def}
  Let $M=(E,r)$ be a matroid, and let $\preceq$ be a total ordering on
  $E$. The \textit{initial matroid} of $M$ with respect to $\preceq$
  is the matroid $\init_{\preceq}(M)=(E,r')$ whose circuits are
  $\{\min_{\preceq}(c):\text{$c$ a circuit of $M$}\}.$ (It is a
  straightforward exercise to check that these circuits define a
  matroid. In fact, $\init_{\preceq}(M)$ is a \textit{discrete
    matroid}, i.e. every element of $E$ is either a loop or a coloop of
  $\init_{\preceq}(M)$.)

  If $\mathscr{M}$ is a tropical ideal, and $\preceq$ is a monomial
  order, then the \textit{initial tropical ideal} of $M$ is defined by
  $\init_{\preceq}(\mathscr{M})_d=(\init_{\preceq}(\mathscr{M}_d)).$
\end{Def}
\begin{notation}\label{not:ColorCode}
  When $r=2,$ there is a natural term order $x\preceq y$. When drawing
  a tropical ideal $\mathscr{M}$, we color-code each monomial $m$ as
  follows:
  \begin{itemize}
  \item Blue and horizontally striped if $m$ is \textit{not} a circuit of
    $\init_{\preceq}(\mathscr{M})$, and
  \item Red and vertically striped if $m$ is \textit{not} a circuit of
    $\init_{\succeq}(\mathscr{M})$.
  \end{itemize}
  More simply, a box is blue if it does not contain the bottom-right-most
  dot of any line segment, and red if it does not contain the
  top-left-most dot of any line segment.
\end{notation}

The \textit{$T$-graph problem} (see \cite{AltmannSturmfels2005}) asks
which pairs of $T$-fixed points in $\Hilb^h_{\mathbf{a}}(\A^r)$ are
endpoints of a 1-dimensional $T$-orbit. The problem has been studied
extensively
\cite{Iarrobino1972,Evain2002,Evain2004,AltmannSturmfels2005,HeringMaclagan2012}. An
algebraic algorithm via Gr\"{o}bner theory for generating the
$T$-graph was given in \cite{AltmannSturmfels2005} and later
implemented as the \texttt{TEdges} Macaulay2 package
\cite{Maclagan2009}; given two monomial ideals $M_1$ and $M_2,$ the
algorithm produces equations that cut out the ``edge scheme''
$E(M_1,M_2)\subseteq\Hilb^h_{\mathbf{a}}(\A^r)$ consisting of ideals
$I$ such that $M_1$ and $M_2$ are the endpoints of $T\cdot I.$ By the
observation above, $E(M_1,M_2)$ is a union of matroid strata.

Recall (Example \ref{ex:FiniteEmbed}) that if $h$ has finitely many
nonzero entries, there is an (equivariant) embedding
$\Hilb_{\mathbf{a}}^h(\A^r)\into(\A^r)^{\left[\sum_dh(d)\right]}.$ In
particular, $(\A^r)^{[n]}$ contains every multigraded Hilbert scheme
$\Hilb_{\mathbf{a}}^h(\A^r),$ where $\sum_dh(d)=n$ and $\mathbf{a}$
is arbitrary, as a closed subscheme. If $r=2,$ then any non-monomial
ideal is homogeneous with respect to at most one positive
grading. Thus the $T$-graph problem for $(\A^2)^{[n]}$ is equivalent
to the $T$-graph problem for every multigraded Hilbert scheme
$\Hilb_{\mathbf{a}}^h(\A^2),$ where $\sum_dh(d)=n$.

\begin{ex}
  In Figure \ref{fig:pictures}, the colored boxes signify that the two
  endpoints of $T\cdot I$ are $(x^5,x^3y^2,y^3)$ and $(x^3,xy^4,y^5).$
  Thus $I$ is a point of the edge scheme
  $E\left(\raisebox{-3pt}{
    \begin{tikzpicture}[scale=0.1]
      \filldraw[white] (-.2,-.2) -- (5.2,-.2) -- (5.2,4.2) -- (-.2,4.2) -- cycle;
      \draw (0,0) -- (1,0) -- (1,1) -- (0,1) -- cycle;
      \draw (1,0) -- (2,0) -- (2,1) -- (1,1) -- cycle;
      \draw (2,0) -- (3,0) -- (3,1) -- (2,1) -- cycle;
      \draw (3,0) -- (4,0) -- (4,1) -- (3,1) -- cycle;
      \draw (4,0) -- (5,0) -- (5,1) -- (4,1) -- cycle;
      \draw (0,1) -- (1,1) -- (1,2) -- (0,2) -- cycle;
      \draw (1,1) -- (2,1) -- (2,2) -- (1,2) -- cycle;
      \draw (2,1) -- (3,1) -- (3,2) -- (2,2) -- cycle;
      \draw (3,1) -- (4,1) -- (4,2) -- (3,2) -- cycle;
      \draw (4,1) -- (5,1) -- (5,2) -- (4,2) -- cycle;
      \draw (0,2) -- (1,2) -- (1,3) -- (0,3) -- cycle;
      \draw (1,2) -- (2,2) -- (2,3) -- (1,3) -- cycle;
      \draw (2,2) -- (3,2) -- (3,3) -- (2,3) -- cycle;
    \end{tikzpicture}
  },\raisebox{-5pt}{
    \begin{tikzpicture}[scale=0.1]
      \filldraw[white] (-.2,-.2) -- (3.2,-.2) -- (3.2,5.2) -- (-.2,5.2) -- cycle;
      \draw (0,0) -- (1,0) -- (1,1) -- (0,1) -- cycle;
      \draw (1,0) -- (2,0) -- (2,1) -- (1,1) -- cycle;
      \draw (2,0) -- (3,0) -- (3,1) -- (2,1) -- cycle;
      \draw (0,1) -- (1,1) -- (1,2) -- (0,2) -- cycle;
      \draw (1,1) -- (2,1) -- (2,2) -- (1,2) -- cycle;
      \draw (2,1) -- (3,1) -- (3,2) -- (2,2) -- cycle;
      \draw (0,2) -- (1,2) -- (1,3) -- (0,3) -- cycle;
      \draw (1,2) -- (2,2) -- (2,3) -- (1,3) -- cycle;
      \draw (2,2) -- (3,2) -- (3,3) -- (2,3) -- cycle;
      \draw (0,3) -- (1,3) -- (1,4) -- (0,4) -- cycle;
      \draw (1,3) -- (2,3) -- (2,4) -- (1,4) -- cycle;
      \draw (2,3) -- (3,3) -- (3,4) -- (2,4) -- cycle;
      \draw (0,4) -- (1,4) -- (1,5) -- (0,5) -- cycle;
    \end{tikzpicture}
  } \right)\subseteq\Hilb_{(1,1)}^{(1,2,3,3,3,1,0,\ldots)}(\A^2)\subseteq(\A^2)^{[13]}.$ 
\end{ex}

\begin{ex}\label{ex:Hilb12211}
  One may compute using \texttt{TEdges} that
  $\Hilb^{(1,2,2,1,1,0,0,\ldots)}_{(1,1)}(\A^2)$ is equivariantly
  isomorphic to $\P^1\times\P^1$ with the diagonal action of
  $\mathbbm{k}^*$. On the left in Figure \ref{fig:Hilb12211} is its
  $T$-graph, drawn in thick black lines, and on the right is the
  $T$-graph of $(\A^2)^{[7]},$ with the corresponding edges
  thickened. The gray curves on the left are intended to depict the
  1-dimensional family of $T$-orbits that correspond to the single
  diagonal edge in the $T$-graph. (The edges of the outer rectangle
  correspond to single 1-dimensional $T$-orbits.) There are ten
  matroid strata:
  \begin{itemize}
  \item The four monomial ideals,
  \item The four black edges of the outer rectangle,
  \item The red curve, representing ideals of the form
    $(x^2-c^2y^2,xy^2+cy^3,y^5)$, and
  \item The open stratum, representing ideals of the form
    $(x^2+c_1xy+(c_1c_2-c_2^2)y^2,xy^2+c_2y^3,y^5)$ with $c_1\ne0$.
  \end{itemize}
      \begin{figure}
        \centering
        \begin{tikzpicture}[scale=.8]
          \draw[thick] (0,0) -- (2,-2) -- (6,0);
          \draw[thick] (0,0) -- (4,2) -- (6,0);
          \foreach \x in {-1,-2,...,-9} {
            \draw[opacity=.15] (0,0) to[out=5*\x,in=180-2.86*\x] (6,0);
            \draw[opacity=.15] (6,0) to[out=180+5*\x,in=-2.86*\x] (0,0);
          };
          \draw[thick] (0,0)--(6,0);
          \draw[red] (0,0) to[out=-5*3,in=180+2.86*3] (6,0);
          \draw (0,0) node {
            \begin{tikzpicture}[scale=0.15]
              \filldraw[white] (-.9,-.9) -- (5.9,-.9) -- (5.9,1.9) --
              (2.9,1.9) -- (2.9,2.9) -- (-.9,2.9) -- cycle;
              \draw (0,0) -- (1,0) -- (1,1) -- (0,1) -- cycle;
              \draw (1,0) -- (2,0) -- (2,1) -- (1,1) -- cycle;
              \draw (2,0) -- (3,0) -- (3,1) -- (2,1) -- cycle;
              \draw (3,0) -- (4,0) -- (4,1) -- (3,1) -- cycle;
              \draw (4,0) -- (5,0) -- (5,1) -- (4,1) -- cycle;
              \draw (0,1) -- (1,1) -- (1,2) -- (0,2) -- cycle;
              \draw (1,1) -- (2,1) -- (2,2) -- (1,2) -- cycle;
            \end{tikzpicture}
          };
          \draw (2,-2) node {
            \begin{tikzpicture}[scale=0.15]
              \filldraw[white] (-.9,-.9) -- (5.9,-.9) -- (5.9,2.1) --
              (1.9,2.1) -- (1.9,3.9) -- (-.9,3.9) -- cycle;
              \draw (0,0) -- (1,0) -- (1,1) -- (0,1) -- cycle;
              \draw (1,0) -- (2,0) -- (2,1) -- (1,1) -- cycle;
              \draw (2,0) -- (3,0) -- (3,1) -- (2,1) -- cycle;
              \draw (3,0) -- (4,0) -- (4,1) -- (3,1) -- cycle;
              \draw (4,0) -- (5,0) -- (5,1) -- (4,1) -- cycle;
              \draw (0,1) -- (1,1) -- (1,2) -- (0,2) -- cycle;
              \draw (0,2) -- (1,2) -- (1,3) -- (0,3) -- cycle;
            \end{tikzpicture}
          };
          \draw (4,2) node {
            \begin{tikzpicture}[scale=0.15]
              \filldraw[white] (-.9,-.5) -- (3.9,-.5) -- (3.9,5.9) -- (-.9,5.9) -- cycle;
              \draw (0,0) -- (1,0) -- (1,1) -- (0,1) -- cycle;
              \draw (1,0) -- (2,0) -- (2,1) -- (1,1) -- cycle;
              \draw (2,0) -- (3,0) -- (3,1) -- (2,1) -- cycle;
              \draw (0,1) -- (1,1) -- (1,2) -- (0,2) -- cycle;
              \draw (0,2) -- (1,2) -- (1,3) -- (0,3) -- cycle;
              \draw (0,3) -- (1,3) -- (1,4) -- (0,4) -- cycle;
              \draw (0,4) -- (1,4) -- (1,5) -- (0,5) -- cycle;
            \end{tikzpicture}
          };
          \draw (6,0) node {
            \begin{tikzpicture}[scale=0.15]
              \filldraw[white] (-.9,-.9) -- (2.9,-.9) -- (2.9,5.9) -- (-.9,5.9) -- cycle;
              \draw (0,0) -- (1,0) -- (1,1) -- (0,1) -- cycle;
              \draw (1,0) -- (2,0) -- (2,1) -- (1,1) -- cycle;
              \draw (0,1) -- (1,1) -- (1,2) -- (0,2) -- cycle;
              \draw (1,1) -- (2,1) -- (2,2) -- (1,2) -- cycle;
              \draw (0,2) -- (1,2) -- (1,3) -- (0,3) -- cycle;
              \draw (0,3) -- (1,3) -- (1,4) -- (0,4) -- cycle;
              \draw (0,4) -- (1,4) -- (1,5) -- (0,5) -- cycle;
            \end{tikzpicture}
          };
        \end{tikzpicture}
        \quad
        \includegraphics[width=4in]{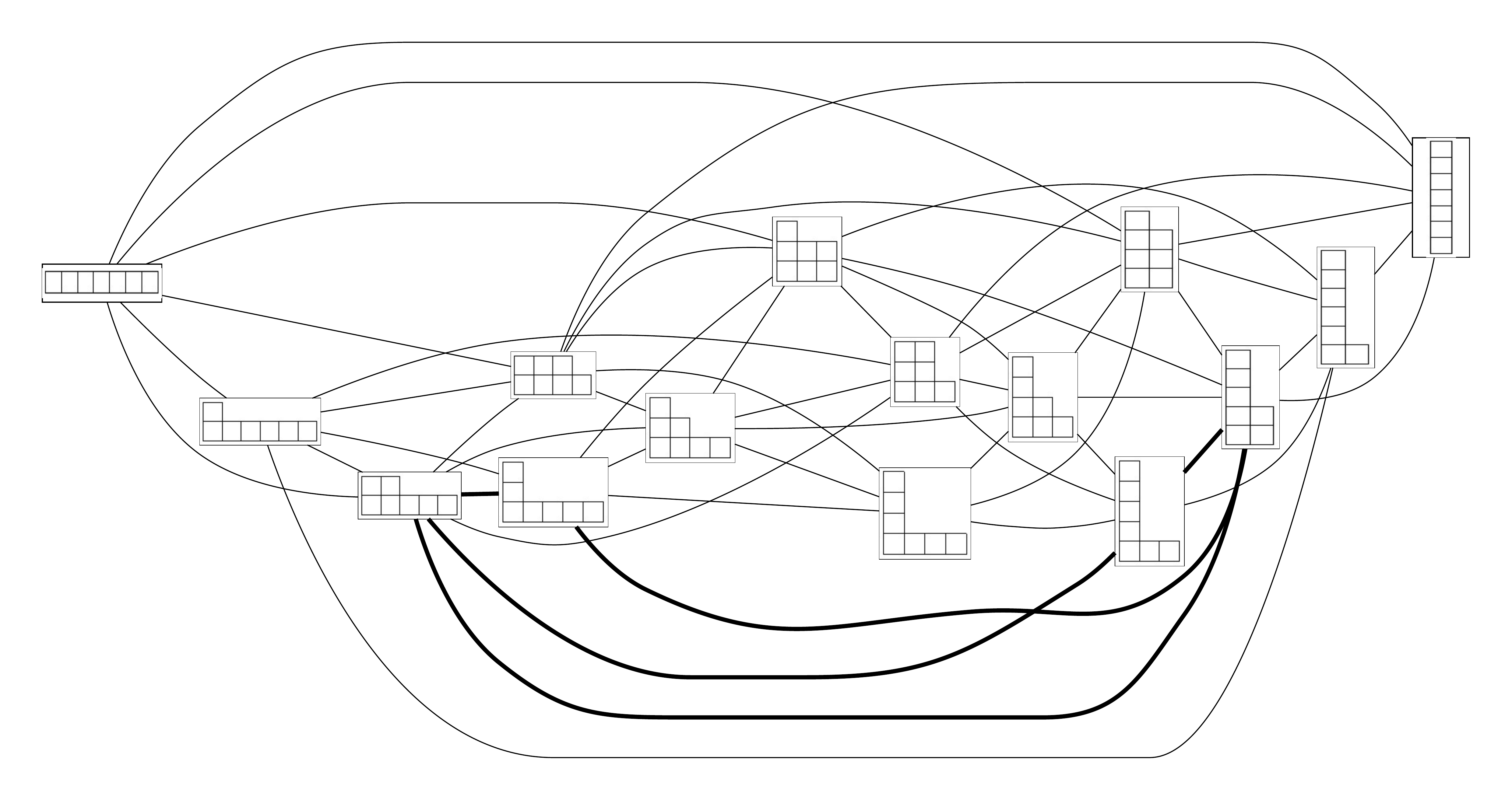}
        \caption{The $T$-graphs of
          $\Hilb_{(1,1)}^{(1,2,2,1,1,0,\ldots)}$ and $(\A^2)^{[7]}$}
        \label{fig:Hilb12211}
      \end{figure}
    \end{ex}
    By viewing $T$-orbit-closures as rational curves in
    $\Hilb^h_{\mathbf{a}}(\A^r)$, and using machinery of
    \textit{unbroken stable maps} \cite{OkounkovPandharipande2010},
    one may associated to each multigraded Hilbert scheme (or edge
    scheme, or intersection of dependence loci) a moduli space
    $\Mbar_{\mathbf{a}}^h(\A^r),$ which roughly parametrizes
    $T$-orbit-closures and their degenerations. (More specifically,
    the moduli space parametrizes $T$-invariant maps
    $f:C\to\Hilb^h_{\mathbf{a}}(\A^r)$, possibly ramified, from nodal
    rational curves to $\Hilb^h_{\mathbf{a}}(\A^r)$, such that $f$ is
    locally $T$-equivariantly smoothable at every node of $C$.)
    \begin{ex}
      In Example \ref{ex:Hilb12211},
      $\Mbar_{(1,1)}^{(1,2,2,1,1,0,\ldots)}(\A^2)
      \cong\P^1,$ with two points corresponding to the two
      degenerations of orbits into nodal rational curves (unions of
      orbits). Another example
      is 
      $\Mbar_{(1,1)}^{(1,2,1,0,\ldots)}(\A^2)
      \cong\P(2,1),$ a weighted projective stack. The orbifold point
      corresponds to a family of 1-dimensional orbits whose limit is a
      doubled line. (In fact, $\Hilb_{(1,1)}^{(1,2,1,0,\ldots)}$ is
      equivariantly isomorphic to $\P^2$ with the
      $\mathbbm{k}^*$-action
      $\lambda\cdot[x:y:z]=[\lambda x:\lambda^{-1}y:z].$ The orbits
      are conics $xy=cz^2$, and the doubled line above is the limit
      $c\to\infty$.)
    \end{ex}

\begin{ex}
  Consider again Example \eqref{eq:HilbPrin}. and let
  $\gamma\in N_{d,k}.$ Suppose $\gamma$ has order $d'$ rotational
  symmetry, where $d'|d.$. Then the element $\zeta_d^{d/d'}\in\C^*=T$
  acts trivially on the $T$-orbit in
  $\mathcal{D}(\{x^d,y^d\})\subseteq(\P^1)^{[k]}$ associated to
  $\gamma$, and $T$ acts with weight $d'$ on this orbit. It follows
  that the moduli space $\Mbar$ associated to
  $\mathcal{D}(\{x^d,y^d\})$ contains a single orbifold point with
  isotropy group $\Z/d'\Z$ corresponding to $\gamma.$ Altogether,
  $\Mbar$ is isomorphic (as a stack) to the \textit{moduli space of
    necklaces}
  $\mathcal{N}_{d,k}=\left[\binom{\{1,\ldots,d\}}{k}/(\Z/d\Z)\right].$
\end{ex}
\begin{ex}
  Using \texttt{TEdges}, we compute that the moduli space associated
  to the $E\left(\raisebox{-4pt}{
    \begin{tikzpicture}[scale=0.1]
      \filldraw[white] (-.2,-.2) -- (5.2,-.2) -- (5.2,4.2) -- (-.2,4.2) -- cycle;
      \draw (0,0) -- (1,0) -- (1,1) -- (0,1) -- cycle;
      \draw (1,0) -- (2,0) -- (2,1) -- (1,1) -- cycle;
      \draw (2,0) -- (3,0) -- (3,1) -- (2,1) -- cycle;
      \draw (3,0) -- (4,0) -- (4,1) -- (3,1) -- cycle;
      \draw (4,0) -- (5,0) -- (5,1) -- (4,1) -- cycle;
      \draw (0,1) -- (1,1) -- (1,2) -- (0,2) -- cycle;
      \draw (1,1) -- (2,1) -- (2,2) -- (1,2) -- cycle;
      \draw (0,2) -- (1,2) -- (1,3) -- (0,3) -- cycle;
      \draw (0,3) -- (1,3) -- (1,4) -- (0,4) -- cycle;
    \end{tikzpicture}
  },\raisebox{-5pt}{
    \begin{tikzpicture}[scale=0.1]
      \filldraw[white] (-.2,-.2) -- (3.2,-.2) -- (3.2,5.2) -- (-.2,5.2) -- cycle;
      \draw (0,0) -- (1,0) -- (1,1) -- (0,1) -- cycle;
      \draw (1,0) -- (2,0) -- (2,1) -- (1,1) -- cycle;
      \draw (2,0) -- (3,0) -- (3,1) -- (2,1) -- cycle;
      \draw (0,1) -- (1,1) -- (1,2) -- (0,2) -- cycle;
      \draw (1,1) -- (2,1) -- (2,2) -- (1,2) -- cycle;
      \draw (2,1) -- (3,1) -- (3,2) -- (2,2) -- cycle;
      \draw (0,2) -- (1,2) -- (1,3) -- (0,3) -- cycle;
      \draw (0,3) -- (1,3) -- (1,4) -- (0,4) -- cycle;
      \draw (0,4) -- (1,4) -- (1,5) -- (0,5) -- cycle;
    \end{tikzpicture}
  } \right)$ is a single point. 
However, this moduli space has ``empty interior,'' in the sense that $E\left(\raisebox{-4pt}{
    \begin{tikzpicture}[scale=0.1]
      \filldraw[white] (-.2,-.2) -- (5.2,-.2) -- (5.2,4.2) -- (-.2,4.2) -- cycle;
      \draw (0,0) -- (1,0) -- (1,1) -- (0,1) -- cycle;
      \draw (1,0) -- (2,0) -- (2,1) -- (1,1) -- cycle;
      \draw (2,0) -- (3,0) -- (3,1) -- (2,1) -- cycle;
      \draw (3,0) -- (4,0) -- (4,1) -- (3,1) -- cycle;
      \draw (4,0) -- (5,0) -- (5,1) -- (4,1) -- cycle;
      \draw (0,1) -- (1,1) -- (1,2) -- (0,2) -- cycle;
      \draw (1,1) -- (2,1) -- (2,2) -- (1,2) -- cycle;
      \draw (0,2) -- (1,2) -- (1,3) -- (0,3) -- cycle;
      \draw (0,3) -- (1,3) -- (1,4) -- (0,4) -- cycle;
    \end{tikzpicture}
  },\raisebox{-5pt}{
    \begin{tikzpicture}[scale=0.1]
      \filldraw[white] (-.2,-.2) -- (3.2,-.2) -- (3.2,5.2) -- (-.2,5.2) -- cycle;
      \draw (0,0) -- (1,0) -- (1,1) -- (0,1) -- cycle;
      \draw (1,0) -- (2,0) -- (2,1) -- (1,1) -- cycle;
      \draw (2,0) -- (3,0) -- (3,1) -- (2,1) -- cycle;
      \draw (0,1) -- (1,1) -- (1,2) -- (0,2) -- cycle;
      \draw (1,1) -- (2,1) -- (2,2) -- (1,2) -- cycle;
      \draw (2,1) -- (3,1) -- (3,2) -- (2,2) -- cycle;
      \draw (0,2) -- (1,2) -- (1,3) -- (0,3) -- cycle;
      \draw (0,3) -- (1,3) -- (1,4) -- (0,4) -- cycle;
      \draw (0,4) -- (1,4) -- (1,5) -- (0,5) -- cycle;
    \end{tikzpicture}
  } \right)$ is actually empty, and the point in question corresponds
to the nodal union of two $T$-orbits, with the node mapping to
$ \raisebox{-5pt}{\begin{tikzpicture}[scale=0.1] \filldraw[white]
    (-.2,-.2) -- (4.2,-.2) -- (4.2,5.2) -- (-.2,5.2) -- cycle; \draw
    (0,0) -- (1,0) -- (1,1) -- (0,1) -- cycle; \draw (1,0) -- (2,0) --
    (2,1) -- (1,1) -- cycle; \draw (2,0) -- (3,0) -- (3,1) -- (2,1) --
    cycle; \draw (3,0) -- (4,0) -- (4,1) -- (3,1) -- cycle; \draw
    (0,1) -- (1,1) -- (1,2) -- (0,2) -- cycle; \draw (1,1) -- (2,1) --
    (2,2) -- (1,2) -- cycle; \draw (0,2) -- (1,2) -- (1,3) -- (0,3) --
    cycle; \draw (0,3) -- (1,3) -- (1,4) -- (0,4) -- cycle; \draw
    (0,4) -- (1,4) -- (1,5) -- (0,5) -- cycle;
    \end{tikzpicture}}.$
\end{ex}

\begin{question}
  The moduli spaces defined above have essentially not been
  studied. We ask, for example: Is $\Mbar_{\mathbf{a}}^h(\A^2)$ smooth
  (as a stack) for all $\mathbf{a}$ and $h$? Rational? What about the
  moduli spaces associated to edge schemes? (From Example
  \ref{ex:Necklace1}, these may be disconnected.)

  Note that in light of Mn\"{e}v's universality theory, the moduli
  spaces associated to arbitrary matroid strata-closures are expected
  to be arbitrarily badly-behaved.
\end{question}

\section{Applications to finite-length Hilbert
  schemes}\label{sec:MonoprincipalIdeals}
Finally, we give a way to apply Theorem \ref{thm:SchurPolynomials} to
the $T$-graph problem for $\Hilb_{\mathbf{a}}^h(\A^2).$ First we need
the following, which is quite useful for working with initial ideals.
\begin{lem}\label{lem:ConvexHull}
  Let $M=(E,r_M)$ be a matroid, and let $\preceq$ be a total order on
  $E$. Let $B_{\preceq}(M)$ be the set of coloops of the (discrete)
  initial matroid $\init_{\preceq}(M)$ (in other words,
  $B_{\preceq}(M)$ is the unique basis for $\init_{\preceq}(M)$), and
  let $B_{\succeq}(M)$ be the set of coloops of $\init_{\succeq}(M)$.

  Let $m\in E,$ and suppose
  \begin{align}\label{eq:MovePast}
    \abs{\{m'\in B_\preceq(M):m'\preceq
    m\}}-&\abs{\{m'\in B_\succeq(M):m'\preceq m\}}\\
    &\quad\quad\le\abs{\{m'\in B_\preceq(M):m'\succeq m\}}-\abs{\{m'\in B_\succeq(M):m'\succeq m\}}.\nonumber
  \end{align}
  Then $m$ is either a loop or a coloop of $M$.
\end{lem}
\begin{proof}
  The following are easy to check using matroid contraction and
  deletion operations:
  \begin{align*}
    \abs{\{m'\in B_\preceq(M):m'\preceq m\}}&=r(\{m'\in E:m'\preceq
                                                 m\})\\
    \abs{\{m'\in B_\succeq(M):m'\preceq m\}}&=r(E)-r(\{m'\in
                                                 E:m'\succ
                                                 m\})\\
    \abs{\{m'\in B_\preceq(M):m'\succeq m\}}&=r(E)-r(\{m'\in
                                                 E:m'\prec
                                                 m\})\\
    \abs{\{m'\in B_\succeq(M):m'\succeq m\}}&=r(\{m'\in E:m'\succeq m\}).
  \end{align*}
  Note that
  $r(\{m'\in E:m'\preceq m\})+r(\{m'\in E:m'\succ m\})\ge r(E),$ with
  equality if and only if $M$ is the direct sum of matroids on the
  groundsets $\{m'\in E:m'\preceq m\}$ and $\{m'\in E:m'\succ
  m\}$. Similarly,
  $r(\{m'\in E:m'\prec m\})+r(\{m'\in E:m'\succeq m\})\ge r(E),$ with
  equality if and only if $M$ is the direct sum of matroids on the
  groundsets $\{m'\in E:m'\prec m\}$ and $\{m'\in E:m'\succeq
  m\}$. Thus the left side of \eqref{eq:MovePast} is nonnegative, the
  right side is nonpositive, and both are zero if and only if $M$ is a
  direct sum of matroids on the groundsets $\{m'\in E:m'\prec m\}$,
  $\{m\}$, and $\{m'\in E:m'\succ m\}$. Thus $m$ is either a
  loop (if the summand on $\{m\}$ has rank zero), or a coloop (if that
  summand has rank 1).
\end{proof}
We apply Theorem \ref{thm:SchurPolynomials} via the observation that
one can obtain a finite-colength ideal from a principal ideal by
adding an appropriate monomial ideal: if $I$ is
$\mathbf{a}$-homogeneous, and $N$ is a monomial ideal, then $I+N$ is
also $\mathbf{a}$-homogeneous. (Of course, not all finite-colength
ideals can be obtained this way, e.g. $(x^2-xy,xy-y^2,x^3)$
cannot.) 
\begin{ex}\label{ex:AddingMonomialsDoesNotCommute}
  Consider $I=(x^2-y^2)\in(\P^1)^{[2]}.$ If $N=(x^3)$, then
  $I+N=(x^2-y^2,x^3)$ is an ideal of colength 6, with tropicalization
  shown in Figure \ref{fig:AddMonomials}. Note that adding $N$ does
  not commute with taking initial ideals; for example,
  $\init_{x<y}(I)+N=(x^2)+(x^3)=(x^2),$ while
  $\init_{x<y}(I+N)=(x^2,xy^2,y^4).$
  \begin{figure}
    \centering
    \includegraphics[height=1.5in]{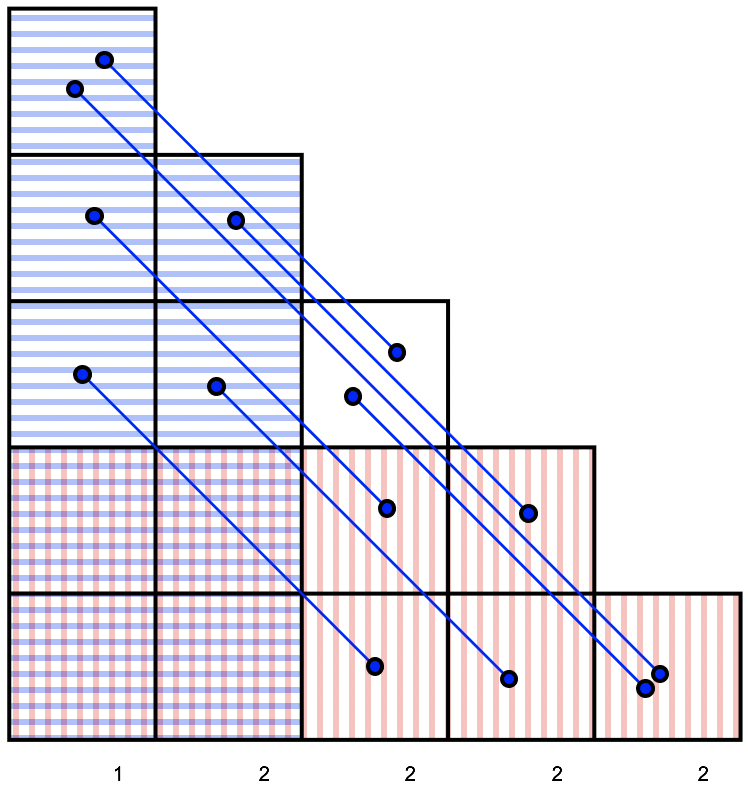}\quad\quad\quad\quad
    \includegraphics[height=1.5in]{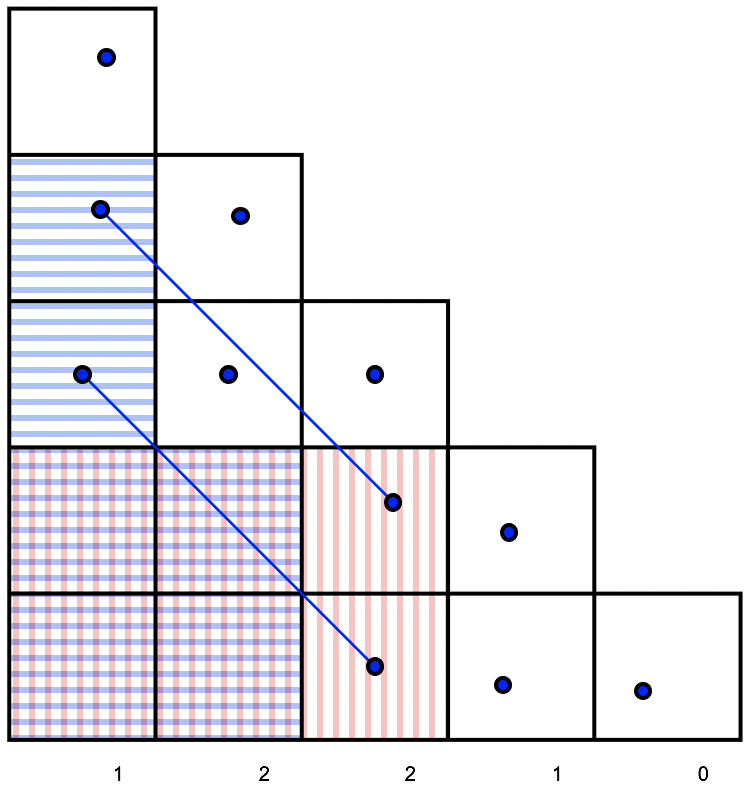}
    \caption{Tropicalizations of $I=(x^2-y^2)$ (left) and $I+(x^3)$ (right).}
    \label{fig:AddMonomials}
  \end{figure}
\end{ex}
\begin{Def}
  A homogeneous ideal $I\subseteq R$ is \textit{PPM} (short for
  \textit{principal plus monomial}) if $I=(f)+N$ for some $f\in R$
  homogeneous, and some monomial ideal $N$.
\end{Def}
The analogous operation of matroids is the ``looped contraction.'' (We
do not know of a standard term for this operation.)
\begin{Def}\label{Def:LoopedContraction}
  Let $M=(E,r)$ be a matroid, and let $S\subseteq E.$ The
  \textit{contraction} $M/S$ of $M$ at $S$ is the matroid with
  groundset $E\setminus S$ whose circuits are the minimal elements of
  $\{S'\cap(E\setminus S):\text{$S'$ a circuit of $M$}\}$. In other
  words, for $T\subseteq E\setminus S,$
  $r_{M/S}(T)=r_M(S\cup T)-r_M(S).$

  The \textit{looped contraction} $M\div S$ is the matroid
  $M\div S=M/S\oplus U_{0,S},$ where $U_{0,S}$ is the uniform matroid
  from Example \ref{ex:UniformMatroid}. Note $M\div S$ has groundset
  $E$ and rank $r(M)-r(S)$, and elements of $S$ are loops in
  $M\div S$. The rank function is given, for $T\subseteq E,$
  $r_{M\div S}(T)=r_M(S\cup T)-r_M(S)$.

  Let $\mathscr{M}$ be a tropical ideal, and let $N$ be a monomial
  ideal. The \textit{looped contraction} $\mathscr{M}\div N$ of
  $\mathscr{M}$ at $N$ is the tropical ideal defined by
  $(\mathscr{M}\div N)_d=\mathscr{M}_d\div S_d$, where $S_d$ is the
  set of monomials in $N_d$. It is straightforward to check that
  $\mathscr{M}\div N$ is a tropical ideal.
  
  A tropical ideal $\mathscr{M}$ is \textit{tropically principal} if
  it has the Hilbert function of a principal ideal, i.e. if there
  exists $c\in\Z_{\ge0}^b$ such that
  $$\rk(\mathscr{M}_d)=\abs{\Mon_d(\mathbf{a})}-\abs{\Mon_{d-c}(\mathbf{a})}$$
  for all $d\in\Z_{\ge0}^b$. (We say $\mathscr{M}$ is
  \textit{generated in degree $c$}.) A tropical ideal $\mathscr{M}$ is
  \textit{PPM} if there exists a tropically principal tropical ideal
  $\mathscr{M}'$ and a monomial ideal $N$ such that
  $\mathscr{M}=\mathscr{M'}\div N.$
\end{Def}
A straightforward calculation using the rank function in Definition
\ref{Def:LoopedContraction} yields:
\begin{prop}\label{prop:AddMonomialsContract}
  For any homogeneous ideal $I$ and any monomial ideal
  $N$, 
  $\Trop(I+N)=\Trop(I)\div N.$
\end{prop}
\begin{cor}
  Let $\mathscr{M}$ be a tropically principal tropical ideal, and let
  $N$ be a monomial ideal. Then $I\mapsto I+N$ defines a morphism
  $\mathcal{S}(\mathscr{M})\to\mathcal{S}(\mathscr{M}\div N)$. (Note
  that $\mathcal{S}(\mathscr{M}\div N)$ lies in a single multigraded
  Hilbert scheme.)
\end{cor}
\begin{cor}\label{cor:TropPPM}
  Let $J\subseteq R$ be an ideal. If $J$ is PPM, then $\Trop(J)$ is
  PPM.
\end{cor}
The converse of Corollary \ref{cor:TropPPM} does not hold:
\begin{ex}
  Let $\mathbbm{k}=\C.$ The tropical ideal $\mathscr{M}$ in Figure
  \ref{fig:PPMCounterexample} is PPM, since
  $\mathscr{M}=\Trop((f)+N)$, where $f=x^3+x^2y+2xy^2+y^3$ and
  $N=(x^4,x^3y^2,x^2y^3,y^4)$. (It is straightforward to check that
  the roots of $f$ do not differ by 4th roots of unity, hence
  $f\not\in\mathcal{D}(\{x^4,y^4\})$. This implies that
  $\Trop((f)+N)_4$ has rank 1, as shown.) On the other hand, we also
  have $\mathscr{M}=\Trop(((x-y)(x-iy)(x+y),x^3y+2x^2y^2)+N),$ as
  follows. Since $((x-y)(x-iy)(x+y))\in\mathcal{D}(\{x^4,y^4\})$,
  $\Trop(((x-y)(x-iy)(x+y))+N)_4$ has rank 2, and adding
  $x^3y+2x^2y^2$ reduces the rank to 1. (It is again easy to check
  that neither ideal contains any extra monomials in degree 4.)

  Lastly, we observe that $((x-y)(x-iy)(x+y),x^3y+2x^2y^2)+N$ is not
  PPM. If it were, it would necessarily be generated in
  degree 4 by $\{x(x-y)(x-iy)(x+y),y(x-y)(x-iy)(x+y),x^4,y^4\}$; these
  span too small a subspace.
  \begin{figure}
    \centering
    \includegraphics[height=1.5in]{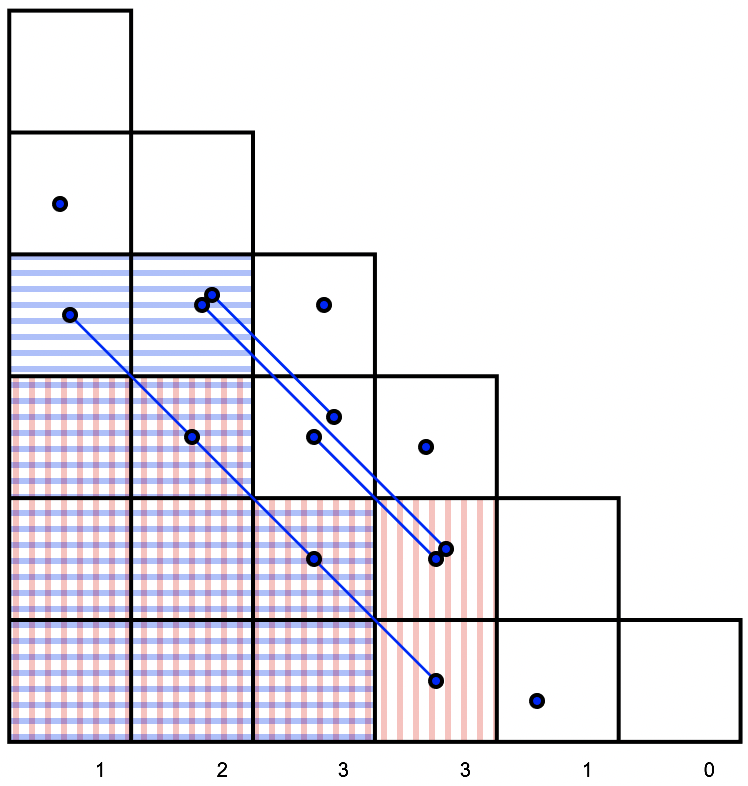}
    \caption{A tropical ideal that is the tropicalization of both a
      PPM ideal and a non-PPM ideal.}
    \label{fig:PPMCounterexample}
  \end{figure}
\end{ex}

The following is the key observation for applying Section \ref{sec:P1}
to Hilbert schemes of finite-length subschemes.
\begin{lem}\label{lem:ConstraintPreserved}
  Let $I=(f)+N$ be a PPM ideal. Let $U\subseteq N_d$ be a
  set of monomials such that
  $\abs{U}>r_{\Trop((f))}(\Mon_d(a_1,a_2))-r_{\Trop(I)}(\Mon_d(a_1,a_2)).$
  Then $(f)\in\mathcal{D}(U).$
\end{lem}
\begin{proof}
  By Proposition \ref{prop:AddMonomialsContract},
  \begin{align*}
    r_{\Trop(I)}(\Mon_d(a_1,a_2))=r_{\Trop((f))}(\Mon_d(a_1,a_2))-r_{\Trop((f))}(N_d).
  \end{align*}
  By assumption,
  $$\abs{U}>r_{\Trop((f))}(\Mon_d(a_1,a_2))-r_{\Trop(I)}(\Mon_d(a_1,a_2))=r_{\Trop((f))}(N_d)\ge
  r_{\Trop((f))}(U),$$ so $U$ is dependent.
\end{proof}

\begin{rem}
  One can apply Lemma \ref{lem:ConstraintPreserved} as follows. Often,
  it can be argued that a given matroid stratum (or edge scheme)
  $\mathcal{S}$ must consist \textit{only} of PPM ideals. In this
  case, recording the nonmonomial generator (where monomials in $N$
  are given coefficient zero) defines a natural embedding from
  $\mathcal{S}$ to a Hilbert scheme
  $\Hilb_{\mathbf{a}}^h(\P^{r-1})\cong\P^N$ of \textit{principal}
  ideals. Lemma \ref{lem:ConstraintPreserved} then says that the
  embedding factors through
  $\bigcap_U\mathcal{D}(U)\subseteq\Hilb_{\mathbf{a}}^h(\P^{r-1}),$
  where $U$ runs over sets satisfying the condition in the hypothesis.
\end{rem}
We conclude by illustrating this method in our running example, Example
\ref{ex:Necklace1}.
\begin{thm}\label{thm:Necklace3}
  Let $k\ge1$ and $d_0>k$. Let $M_1$ (resp. $M_2$) be the partition
  whose Young diagram is an $d_0\times k$ (resp. $k\times d_0$)
  rectangle. Then the edge scheme
  $E(M_1,M_2)\subseteq(\A^2)^{[d_0\cdot k]}$ is isomorphic to
  $\mathcal{D}(\{x^{d_0},y^{d_0}\})\subseteq(\P^1)^{[k]}$, i.e. it
  consists of a collection of rational curves, indexed by necklaces
  with $k$ black and $d_0-k$ white beads, all of which meet at two
  points.
\end{thm}
\begin{proof}
  First, we argue that any ideal $I$ in the edge scheme $E(M_1,M_2)$
  is PPM, with nonmonomial generator in degree $k$. Note that the
  Hilbert function of $M_1$ and $M_2$ with respect to the grading
  $(1,1)$ is
  $$(1,2,\ldots,\underbrace{k,k,\ldots,k}_{d_0-k+1},k-1,\ldots,1,0,0,\ldots)=
  \begin{cases}
    d+1&0\le d\le k-1\\
    k&k\le d\le d_0-1\\
    d_0+k-1-d&d_0\le d\le d_0+k-1\\
    0&d>d_0+k-1.
  \end{cases}
  $$ For $d\le k-1,$ $\dim I_d=0$. For $k\le d\le d_0-1,$
  $\dim I_d=d+1-k,$ which implies $I_d$ is spanned by the $d+1-k$
  linearly independent monomial multiples of the generator of $I_k$.

  For $d_0\le d\le d_0+k-1$, $I_d$ contains the $d+1-k$ monomial
  multiples of the generator of $I_k$, as well as the $d-d_0+1$
  consecutive monomials $x^{d},x^{d-1}y,\ldots,x^{d_0}y^{d-d_0}$, by
  Lemma \ref{lem:ConvexHull}. Since $d-d_0<k$, by an
  upper-triangularity argument, these $(d+1-k)+(d-d_0+1)$ vectors are
  all linearly independent. On the other hand,
  $$\dim I_d=2d-k-d_0+2=(d+1-k)+(d-d_0+1).$$ In particular, if $f$ is
  a generator of $I_k$, we have shown that $I=(f)+(x^{d_0}),$ hence is
  PPM. (This is from the case $d=d_0$.)

  Next, we apply Lemma \ref{lem:ConstraintPreserved}. Again by Lemma
  \ref{lem:ConvexHull}, $y^{d_0}\in I,$ so we may as well write
  $I=(f)+(x^{d_0},y^{d_0}).$ Let $U=\{x^{d_0},y^{d_0}\}$, and note that
  $$2=\abs{U}>r_{\Trop((f))}(\Mon_{d_0})-r_{\Trop(I)}(\Mon_{d_0})=k-(k-1)=1.$$
  By Lemma \ref{lem:ConstraintPreserved}, $f\in\mathcal{D}(U).$ This
  shows that $\mathcal{M}(M_1,M_2)\subseteq\mathcal{N}_{d_0,k}$, and
  the opposite inclusion follows immediately from counting ranks in
  each grade.
\end{proof}
  Theorem \ref{thm:Necklace3} immediately generalizes, with the same
  proof, to the case where $M_1$ and $M_2$ are both ``cut off'' in
  some degree $d_1>d_0.$ For example,
  $\M(M_1,M_2)\cong \mathcal{N}_{6,4},$ where
  \begin{align*}
    M_1&=\raisebox{-4pt}{
    \begin{tikzpicture}[scale=0.1]
      \filldraw[white] (-.2,-.2) -- (6.2,-.2) -- (6.2,4.2) -- (-.2,4.2) -- cycle;
	 \draw (0,0) -- (1,0) -- (1,1) -- (0,1) -- cycle;
	 \draw (1,0) -- (2,0) -- (2,1) -- (1,1) -- cycle;
	 \draw (2,0) -- (3,0) -- (3,1) -- (2,1) -- cycle;
	 \draw (3,0) -- (4,0) -- (4,1) -- (3,1) -- cycle;
	 \draw (4,0) -- (5,0) -- (5,1) -- (4,1) -- cycle;
	 \draw (5,0) -- (6,0) -- (6,1) -- (5,1) -- cycle;
	 \draw (0,1) -- (1,1) -- (1,2) -- (0,2) -- cycle;
	 \draw (1,1) -- (2,1) -- (2,2) -- (1,2) -- cycle;
	 \draw (2,1) -- (3,1) -- (3,2) -- (2,2) -- cycle;
	 \draw (3,1) -- (4,1) -- (4,2) -- (3,2) -- cycle;
	 \draw (4,1) -- (5,1) -- (5,2) -- (4,2) -- cycle;
	 \draw (5,1) -- (6,1) -- (6,2) -- (5,2) -- cycle;
	 \draw (0,2) -- (1,2) -- (1,3) -- (0,3) -- cycle;
	 \draw (1,2) -- (2,2) -- (2,3) -- (1,3) -- cycle;
	 \draw (2,2) -- (3,2) -- (3,3) -- (2,3) -- cycle;
	 \draw (3,2) -- (4,2) -- (4,3) -- (3,3) -- cycle;
	 \draw (4,2) -- (5,2) -- (5,3) -- (4,3) -- cycle;
	 \draw (0,3) -- (1,3) -- (1,4) -- (0,4) -- cycle;
	 \draw (1,3) -- (2,3) -- (2,4) -- (1,4) -- cycle;
	 \draw (2,3) -- (3,3) -- (3,4) -- (2,4) -- cycle;
	 \draw (3,3) -- (4,3) -- (4,4) -- (3,4) -- cycle;
    \end{tikzpicture}
}&M_2&=\raisebox{-6pt}{
      \begin{tikzpicture}[scale=0.1]
        \filldraw[white] (-.2,-.2) -- (4.2,-.2) -- (4.2,6.2) -- (-.2,6.2) -- cycle;
	 \draw (0,0) -- (1,0) -- (1,1) -- (0,1) -- cycle;
	 \draw (1,0) -- (2,0) -- (2,1) -- (1,1) -- cycle;
	 \draw (2,0) -- (3,0) -- (3,1) -- (2,1) -- cycle;
	 \draw (3,0) -- (4,0) -- (4,1) -- (3,1) -- cycle;
	 \draw (0,1) -- (1,1) -- (1,2) -- (0,2) -- cycle;
	 \draw (1,1) -- (2,1) -- (2,2) -- (1,2) -- cycle;
	 \draw (2,1) -- (3,1) -- (3,2) -- (2,2) -- cycle;
	 \draw (3,1) -- (4,1) -- (4,2) -- (3,2) -- cycle;
	 \draw (0,2) -- (1,2) -- (1,3) -- (0,3) -- cycle;
	 \draw (1,2) -- (2,2) -- (2,3) -- (1,3) -- cycle;
	 \draw (2,2) -- (3,2) -- (3,3) -- (2,3) -- cycle;
	 \draw (3,2) -- (4,2) -- (4,3) -- (3,3) -- cycle;
	 \draw (0,3) -- (1,3) -- (1,4) -- (0,4) -- cycle;
	 \draw (1,3) -- (2,3) -- (2,4) -- (1,4) -- cycle;
	 \draw (2,3) -- (3,3) -- (3,4) -- (2,4) -- cycle;
	 \draw (3,3) -- (4,3) -- (4,4) -- (3,4) -- cycle;
	 \draw (0,4) -- (1,4) -- (1,5) -- (0,5) -- cycle;
	 \draw (1,4) -- (2,4) -- (2,5) -- (1,5) -- cycle;
	 \draw (2,4) -- (3,4) -- (3,5) -- (2,5) -- cycle;
	 \draw (0,5) -- (1,5) -- (1,6) -- (0,6) -- cycle;
	 \draw (1,5) -- (2,5) -- (2,6) -- (1,6) -- cycle;
       \end{tikzpicture}
       }.
  \end{align*}

\bibliographystyle{alpha}
\bibliography{research}

\end{document}